\documentclass[oneside]{amsart}

\usepackage{amsmath, amsthm, amsfonts}
\usepackage{amssymb}
\usepackage{latexsym}
\usepackage{verbatim}
\usepackage{url}
\usepackage{textcomp}
\usepackage[all,cmtip]{xy}
\usepackage{color}
\usepackage{hyperref}
\usepackage{tabularx}
\input{xypic}
\usepackage{enumerate}
\usepackage{amssymb}
\usepackage{array}
\usepackage{mathtools}

\newcommand{\etalchar}[1]{$^{#1}$}

\DeclareMathAlphabet{\mathfr}{U}{euf}{m}{n}

\newtheorem{theorem}{Theorem}[section]

\newtheorem{proposition}[theorem]{Proposition}
\newtheorem{corollary}[theorem]{Corollary}

\newtheorem{lemma}[theorem]{Lemma}
\newtheorem{question}[theorem]{Question}

\theoremstyle{remark}
\newtheorem{remark}[theorem]{Remark}
\newtheorem{definition}[theorem]{Definition}

\newcommand\acc[2]{\ensuremath{{}^{#1}\hskip-0.3ex{#2}}}
\newcommand{\comp}{\begin{picture}(6,5)(-3,-2)\put(0,1){\circle{2}} \end{picture}}\def\circ{\comp}

\newcommand{\lra}{\longrightarrow}
\newcommand{\ra}{\rightarrow}

\newcommand{\Q}{\mathbb Q}
\newcommand{\Qbar}{{\overline{\mathbb Q}}}

\newcommand{\Gal}{\mathrm{Gal}}
\newcommand{\R}{\mathbb R}
\newcommand{\Z}{\mathbb Z}

\newcommand{\C}{\mathbb C}
\newcommand{\GL}{\mathrm{GL}}
\newcommand{\SL}{\mathrm{SL}}
\newcommand{\M}{\mathrm{M }}
\newcommand{\PGL}{\mathrm{PGL}}
\newcommand{\Cl}{\mathrm{Cl}}
\newcommand{\End}{\operatorname{End}}
\newcommand{\Hom}{\operatorname{Hom}}

\newcommand{\Jac}{\operatorname{Jac}}

\newcommand{\Br}{\mathrm{Br}}

\newcommand{\Res}{\operatorname{Res}}
\newcommand{\im}{\operatorname{Im}}
\newcommand{\id}{\operatorname{id}}
\newcommand{\Tr}{\operatorname{Tr}}
\newcommand{\Unitary}{\mathrm{U}}
\newcommand{\SU}{\mathrm{SU}}
\newcommand{\USp}{\mathrm{USp}}
\newcommand{\ST}{\mathrm{ST}}
\newcommand{\Aut}{\mathrm{Aut}}
\newcommand{\Mm}{{\mathcal M}}

\numberwithin{equation}{section}
\newcommand{\cB}{\mathcal{B}}

\usepackage{url}

\newcommand{\HH}{\mathbb{H}}
\newcommand{\bF}{\mathbf{F}}
\newcommand{\bE}{\mathbf{E}}
\newcommand{\bD}{\mathbf{D}}
\newcommand{\bC}{\mathbf{C}}
\newcommand{\bB}{\mathbf{B}}
\newcommand{\bA}{\mathbf{A}}

\newcommand{\cyc}[1]{{\mathrm{C}_#1}}
\newcommand{\dih}[1]{{\mathrm{D}_#1}}
\newcommand{\alt}[1]{{\mathrm{A}_#1}}
\newcommand{\sym}[1]{{\mathrm{S}_#1}}

\begin{document}

\title[Elliptic $k$-curves and genus $2$ Sato--Tate groups]{Fields of definition of elliptic $k$-curves and the realizability of all genus 2 Sato--Tate groups over a number field}

\author{Francesc Fit\'e}
\address{Departament de Matem\`atiques, Universitat Polit\`ecnica de Catalunya \\Edifici Omega, C/Jordi Girona 1--3\\
08034 Barcelona, Catalonia
}
\email{francesc.fite@gmail.com}
\urladdr{https://mat-web.upc.edu/people/francesc.fite/}
\thanks{Fit\'e was funded by the German Research Council via SFB/TR 45. }

\author{Xavier Guitart}
\address{Departament d'\`Algebra i Geometria\\
Universitat de Barcelona\\Gran via de les Corts Catalanes, 585\\
08007 Barcelona, Catalonia
}
\email{xevi.guitart@gmail.com}
\urladdr{http://atlas.mat.ub.edu/personals/guitart/}
\thanks{Guitart was partially funded by  MTM2012-33830 and MTM2012-34611.}

\date{\today}

\begin{abstract} Let $A/\Q$ be an abelian variety of dimension $g\geq 1$ that is isogenous over~$\Qbar$ to~$E^g$, where $E$ is an elliptic curve. If~$E$ does not have complex multiplication (CM), by results of Ribet and Elkies concerning fields of definition of elliptic $\Q$-curves $E$ is isogenous to a curve defined over a polyquadratic extension of $\Q$. We show that one can adapt Ribet's methods to study the field of definition of $E$ up to isogeny also in the CM case. We find two applications of this analysis to the theory of Sato--Tate groups: First, we show that~$18$ of the $34$ possible Sato--Tate groups of abelian surfaces over~$\Q$ occur among at most~$51$ $\Qbar$-isogeny classes of abelian surfaces over~$\Q$; Second, we give a positive answer to a question of Serre concerning the existence of a number field over which abelian surfaces can be found realizing each of the $52$ possible Sato--Tate groups of abelian surfaces.  
\end{abstract}
\maketitle

\tableofcontents

\section{Introduction}

It is well known that there exist three possibilities for the Sato--Tate group of an elliptic curve $E$ defined over a number field $k$: the special unitary group $\SU(2)$ of degree $2$, the unitary group $\Unitary(1)$ of degree $1$ embedded in $\SU(2)$, and the normalizer $N(\Unitary(1))$ of $\Unitary(1)$ in $\SU(2)$. These three possibilities are in accordance with the following trichotomy: the elliptic curve does not have complex multiplication (CM), the elliptic curve has CM defined over $k$, and the elliptic curve has CM but not defined over $k$.

From this description, it is easy to see that there exists a number field (for example, take any quadratic imaginary field of class number $1$) over which three elliptic curves can be defined realizing each of the three Sato--Tate groups $\SU(2)$, $\Unitary(1)$, and $N(\Unitary(1))$.

In \cite{FKRS12}, it was shown that there exist 52 possible Sato--Tate groups of abelian surfaces over number fields, all of which occur for some choice of the number field and of the abelian surface defined over it. Let $N_{\mathrm{ST},2}(k)$ denote the number of subgroups of $\USp(4)$ up to conjugacy that arise as Sato--Tate groups of abelian surfaces defined over the number field $k$. In \cite{FKRS12}, it was proven that $N_{\mathrm{ST},2}(\Q)=34$. 
Serre has asked the following question.

\begin{question}\label{question: existence}
Does there exist a number field $k$ over which $52$ abelian surfaces can be found realizing each of the $52$ possible Sato--Tate groups of abelian surfaces over number fields? In other words, does there exist a number field $k$ such that $N_{\mathrm{ST},2}(k)=52$? 
\end{question}

It is well-known that the connected component of the identity of the Sato--Tate group is invariant under base change. However, its group of components is sensitive to base change, and it is therefore not necessarily true (and in fact, it is not) that the compositum of all fields of definition of the examples in \cite{FKRS12} gives a positive answer to the above question.

The question is thus on whether the conditions imposed on $k$ by each of the possible groups of components are compatible or not among them. Fundamental to this analysis is the existence of an isomorphism between the group of components of the Sato--Tate group of an abelian surface $A$ defined over $k$ and the Galois group $\Gal(K/k)$, where $K/k$ is the minimal extension over which all of the endomorphisms of $A$ are defined.

Let us look at a concrete example of the type of conditions on $k$ imposed by certain groups of components. We will consider abelian surfaces $A$ defined over a number field $k$ satisfying
\begin{enumerate}
\item[(P)]  $\Gal(K/k)$ contains $\sym 4$.
\end{enumerate}
One can easily show that (P) implies that $A$ is isogenous over $\Qbar$ to the square of an elliptic curve with CM, say by a quadratic imaginary field $M$. There are only three Sato--Tate groups (denoted by~$O$, $O_1$, and~$J(O)$ in \cite{FKRS12}) whose group of components contains the symmetric group~$\sym 4$.
The dictionary between Sato--Tate groups and Galois endomorphism types of \cite{FKRS12} ensures that~$O$ only arises among $A$ satisfying (P) for which $M\subseteq k$, whereas~$O_1$ and $J(O)$ can occur only for $A$ satisfying (P) with $M\not \subseteq k$.

Up to our knowledge, all the examples in the existing literature of abelian surfaces~$A$ satisfying (P) have $M=\Q(\sqrt{-2})$. The lack of a simple construction of such an~$A$ with $M\not =\Q(\sqrt{-2})$, together with the fact that the wide computational search of abelian surfaces performed in \cite{FKRS12} yielded only examples of $A$ satisfying (P) with $M=\Q(\sqrt{-2})$, may suggest that indeed $M=\Q(\sqrt{-2})$ might be a necessary condition for $A$ to satisfy (P). If this was the case, then Question~\ref{question: existence} would have a negative answer, since in order to realize, for example, both~$O$ and~$J(O)$ the field~$k$ would be forced to contain and not contain $\Q(\sqrt{-2})$ simultaneously.

As a consequence of the preceding discussion, one is naturally led by Question~\ref{question: existence} to the following question regarding a basic aspect of the arithmetic of abelian varieties over $\Q$.

\begin{question}\label{question: sobre M}
Let $A$ be an abelian variety defined over $\Q$ of dimension $g\geq 1$ that is $\Qbar$-isogenous to $E^g$, where $E$ is an elliptic curve over $\Qbar$ with CM by $M$.  
\begin{itemize}
\item[(A)] Which is the set of possibilities for $M$? 
\item[(B)] Let $K/M$ be the minimal extension over which all the endomorphisms of $A$ are defined. Does the prescription of $\Gal(K/M)$ impose further restrictions on $M$?
\end{itemize}
\end{question}

If $g=1$, then the theory of complex multiplication shows that $M$ is a quadratic imaginary field of class number $1$ and thus there are only~$9$ possibilities for $M$. In this work, we provide an answer to Question~\ref{question: sobre M} for $g=2$, which sets the basis on which we build a positive answer to Question~\ref{question: existence}.

\textbf{Main results.} As the case $g=1$ may suggest, an answer to Question~\ref{question: sobre M} will follow, via the theory of complex multiplication, from gaining control on the field of definition of the elliptic factor $E$. The study of the field of definition of $E$ will be carried out in \S\ref{section: field def}, whose main result is the following (see Theorem~\ref{theorem: ggext}). 

\begin{theorem}\label{theorem: intro1}
Let $E$ be an elliptic curve defined over $\Qbar$ with CM by a quadratic imaginary field $M$ such that $E^g$, for some $g\geq 1$, is $\Qbar$-isogenous to an abelian variety $A$ defined over $k$. Assume that $M$ is contained in $k$. Let $K/k$ be the minimal extension over which all the endomorphisms of $A$ are defined. Then $E$ is $\Qbar$-isogenous to an elliptic curve $E^*$ defined over an abelian subextension $F/k$ of $K/k$ such that every element in $\Gal(F/k)$ has order dividing~$g$.
\end{theorem}

Observe that the elliptic curve $E$ in the preceding theorem is an elliptic $k$-curve in the sense of Ribet (the notion of elliptic $k$-curve and, more generally, of abelian $k$-variety, will be recalled in \S \ref{section: kvar}). A characterization of the field of definition of an elliptic $k$-curve $C$ without CM (achieved by Ribet \cite{Rib94} and Elkies \cite{Elk04} by means of completely different methods) is well known: $C$ admits a model up to isogeny defined over a polyquadratic\footnote{That is, the composition of a finite number of quadratic extensions.} extension of $k$. We will show that in our particular situation Ribet's methods can also be applied to the CM case. The crucial idea is that $E^g$ admitting a model over $k$ implies that the $g$th power of the cohomology class $\gamma_{E}\in H^2(G_k,M^*)$ naturally attached to $E$ is trivial, a property that does not necessarily hold for an arbitrary CM elliptic $k$-curve~$C$.

 Combining Theorem~\ref{theorem: intro1} with bounds on the degree of $K/k$ (deduced from the methods of \cite{Sil92} and \cite{GK16}), one obtains the following answer to part $(A)$ of Question~\ref{question: sobre M}.

\begin{theorem}\label{theorem: intro2}
With the notations as in Question~\ref{question: sobre M}, the class group $\Cl(M)$ of~$M$ has exponent dividing $g$. If moreover $g$ is prime, then 
$$
\Cl(M)=
\begin{cases}
1,\cyc g,\text{ or }\cyc g \times \cyc g & \text{ if } g=2,\\ 
1\text{ or } \cyc g & \text{ otherwise.}
\end{cases}
$$
In particular, if $g=2$, there are at most $51$ possibilities for $M$. 
\end{theorem}

We refer the reader to Theorem~\ref{theorem: absurfQ} for a more precise version of the above theorem in the case $g=2$, which accounts also for part $(B)$ of Question~\ref{question: sobre M}. Just to show the flavour of the result, let us state a particular instance of it: if for example $\Gal(K/M)\simeq \sym 4$, then $M$ is either $\Q(\sqrt{-1})$, $\Q(\sqrt{-2})$ or a quadratic imaginary field of class number~$2$ and distinct from $\Q(\sqrt{-15})$, $\Q(\sqrt{-35})$, $\Q(\sqrt{-51})$, and $\Q(\sqrt{-115})$.

These kind of constraints on $M$ are obtained by means of two different methods. First, we consider obstructions coming from group representation theory. These are obtained by pushing a bit further the study in \cite[\S3]{FS14} of the Galois module structure of the ring of endomorphisms of an abelian surface that is $\Qbar$-isogenous to the square of an elliptic curve. This is done in \S\ref{section: galstr}. Second, we consider obstructions coming from group cohomology. These are obtained by means of a detailed study of the endomorphism algebra of the Weil's restriction of scalars $\Res_{K/k}E$ in the case $\Gal(K/k)\simeq \sym 4$. We again benefit from a technique exploited in Ribet's work: we look at this algebra as a twisted group algebra. A key step towards determining its structure is to compute its center; we devote \S\ref{sec: restriction of scalars} to that calculation. In \S\ref{section: imquad}, it only remains to combine the results of \S\ref{section: galstr} and \S\ref{sec: restriction of scalars}. 

Thanks to the dictionary between Sato--Tate groups and Galois endomorphism types, there is an equivalent formulation of Theorem~\ref{theorem: intro2} in the case $g=2$.

\begin{theorem}\label{theorem: intro3} Among the $34$ possibilities for the Sato--Tate group of an abelian surface defined over $\Q$, the $18$ with identity component isomorphic to $\Unitary(1)$ occur among at most $51$ $\Qbar$-isogeny classes of abelian surfaces over $\Q$.
\end{theorem}

We refer to Theorem~\ref{theorem: finnum} for a more precise version of the above statement. At this stage, we are ready to provide an affirmative answer to Question~\ref{question: existence}, which is the main result of this paper (see Theorem~\ref{theorem: main}).

\begin{theorem}
One has $N_{\mathrm{ST},2}(k_0)=52$ for
$$
k_0:=\Q(\sqrt{-40},\sqrt{-51},\sqrt{-163},\sqrt{-67},\sqrt{19\cdot 43},\sqrt{-57})\,.
$$
\end{theorem}

After the discussion following Question~\ref{question: existence}, it is clear that constructing an abelian surface defined over $k$ satisfying (P) and for which $M\not =\Q(\sqrt{-2})$ is a crucial step in the proof. We devote the whole of \S\ref{section: an example} to construct an abelian surface $A$ defined over $k=\Q(\sqrt{-40})$, $\Qbar$-isogenous to the square of an elliptic curve with CM by $M=\Q(\sqrt{-40})$, and for which $\Gal(K/k)\simeq \sym 4$.  This abelian surface $A$ is obtained as a simple factor of the restriction of scalars of a certain elliptic curve with CM by $\Q(\sqrt{-40})$ over a suitable extension $K/k$; in the construction, we make crucial use of the techniques developed in \S\ref{sec: restriction of scalars}. The construction of the remaining abelian surfaces, which is carried out in \S\ref{section: ST}, requires less elaborate techniques, though still forces~$k_0$ to contain a few more additional quadratic fields.

\textbf{Notations and terminology.}
Throughout this article,~$k$ will be a number field assumed to be contained in a fixed algebraic closure $\Qbar$ of~$\Q$. We will write~$G_k$ for the absolute Galois group $\Gal(\Qbar/k)$. All field extensions of~$k$ considered will be assumed to be algebraic and contained in $\Qbar$. By $\zeta_r$ we will denote a primitive $r$th root of unity, and $g$ will be an integer $\geq 1$.  We will work in the category of abelian varieties up to isogeny over $k$. This means that isogenies become invertible in this category and therefore, if~$A$ is an abelian variety defined over~$k$, we will write $\End(A)$ to denote the $\Q$-algebra of endomorphisms defined over $k$ (what some authors write $\End(A)\otimes \Q$ or $\End^0(A)$). Similarly, if~$B$ is an abelian variety defined over~$k$, then $\Hom(A,B)$ will denote the $\Q$-vector space of homomorphisms from~$A$ to~$B$ that are defined over~$k$. Although isogenies will be the isomorphisms of our working category, we will still write $A\sim B$ to mean that~$A$ and~$B$ are isogenous over~$k$. If $L/k$ is a field extension, then $A_L$ will denote the base change of $A$ from $k$ to $L$, and we will write $A_L\sim B_L$ if~$A$ and~$B$ become isogenous over~$L$ (in particular, we will write $\Hom(A_L,B_L)$ to refer to what some authors write as $\Hom_L(A,B)$). If~$C$ is an abelian variety over $L$, for short, we will say that~$C$ admits a model up to isogeny over $k$ if there exists an abelian variety~$C_0$ defined over $k$ such that $C_\Qbar\sim C_{0,\Qbar}$. Finally, if $\mathcal{B}$ is an algebra we will denote by $Z(\mathcal{B})$ its center.

\textbf{Acknowledgments.} Fit\'e is thankful to CIRM and ICERM for invitations in May and in September of 2015, respectively, which facilitated fruitful conversations with Drew Sutherland regarding this work. Thanks to Carlos de Vera, Elisa Lorenzo, and Mark Watkins for helpful comments. We are thankful to Kiran Kedlaya and Jordi Quer for their inspiring suggestions, specially regarding the proof of Corollary~\ref{corollary: classbiquad_new}. 

\textbf{Erratum.} This version contains some corrections to the version published in Trans. Amer. Math. Soc. \textbf{370}, n. 7 (2018). These corrections are related to a mistake in part ii) of Lemma \ref{lemma: LF} in the published version (which was pointed out to us by Kiran Kedlaya, and for which we express him our gratitude). The statement and proof of the lemma have been amended in the present version. In addition to $\mathrm{B_O}$ and $\GL_2(\Z/3\Z)$, the corrected lemma now allows for the complex reflexion group $G_8$. The correction in the lemma carries some consequences: we note that part i) of Proposition \ref{proposition: M}, as well as row $\sym 4$ of Table \ref{table: M}, and row $J(O)$ of Table \ref{table: MM} now allow for the possibility $M=\Q(\sqrt{-1})$. We accordingly corrected the allusion to row $\sym 4$ of Table \ref{table: M} in the Introduction. We note that these corrections do not affect any of the main results of the paper, neither the strategy methods used in their proofs. To emphasize on this latter point, we have also rephrased Remark \ref{remark: nofields} in the present version of the manuscript. 

\section{Fields of definition of elliptic $k$-curves}
In this section we study fields of definition up to isogeny of certain elliptic curves~$E$ with CM. Namely, those with the property that the CM field is contained in $k$ and such that $E^g$ admits a model up to isogeny over $k$. The main tool that we use is to exploit the fact that, in this setting,~$E$ is an elliptic $k$-curve.

\subsection{Preliminary results on $k$-varieties}\label{section: kvar}
In this section we collect some background and preliminary results on abelian $k$-varieties. In fact, we are mainly interested in the $1$-dimensional case of elliptic $k$-curves. The reader can consult  \cite{Rib94}, \cite{Qu00}, and \cite{Py02} as general references for $k$-curves and $k$-varieties.
\begin{definition}[After Ribet]\label{definition: kcurve}
  An abelian variety $B/\Qbar$ is called an abelian $k$-variety (or just a $k$-variety, for short) if for all $\sigma \in G_k$ there exists an isogeny $\mu_\sigma\colon\acc\sigma B \ra B$ which is compatible with $\End(B)$, in the sense that the diagram
\begin{align}\label{eq: compatibility}
  \xymatrix{
    \acc\sigma B \ar[d]^{\acc\sigma \varphi} \ar[r]^{\mu_\sigma} &B \ar[d]^\varphi\\
   \acc\sigma B \ar[r]^{\mu_\sigma}
&B}
\end{align}
commutes for all $\varphi\in \End(B)$.
\end{definition}
A simple calculation shows that the property of being a $k$-variety only depends on the $\Qbar$-isogeny class of $B$. In  fact, if $K$ is an extension of $k$ one says that $B$ is a $k$-variety defined over $K$ if $B$ is defined over $K$ and $B_\Qbar$ is a $k$-variety.

Let $B$ be a $k$-variety with endomorphism algebra $\cB=\End(B)$, and let $R=Z(\cB)$ denote its center. The variety $B$ has a model defined over a number field. By enlarging this field if necessary, we can assume that it is also a field of definition of the isogenies between $B$ and its conjugates. Thus we can (and do) fix a system of compatible isogenies $\{\mu_\sigma\colon \acc\sigma B\ra B\}_{\sigma\in G_k}$ which is locally constant. For every $\sigma,\tau\in G_k$ define 
\begin{align*}
  c_B(\sigma,\tau)=\mu_\sigma\circ\acc\sigma\mu_\tau\circ \mu_{\sigma\tau}^{-1}.
\end{align*}
A short computation shows that $c_B(\sigma,\tau)$ lies in $R^*$, and that the map \[\sigma,\tau\mapsto c_B(\sigma,\tau)\] is a continuous $2$-cocycle of $G_k$ with values in $R^*$ (considered as a $G_k$-module with trivial action). Denote by $\gamma_B\in H^2(G_k,R^*)$ the cohomology class of $c_B$.  To the best of our knowledge this cohomology class was introduced by Ribet in \cite{Rib92}, and it is one of the main tools for studying the arithmetic of $k$-varieties. The next proposition summarizes some of the main properties of $c_B$ and $\gamma_B$.
\begin{proposition}\label{prop: main props k-variety}
  \begin{enumerate}[i)]
\item If $c_B'$ is a cocycle cohomologous to $c_B$, there exist compatible isogenies $\lambda_\sigma\colon\acc\sigma B\ra B$ such that $c_B'(\sigma,\tau)=\lambda_\sigma\circ\acc\sigma\lambda_\tau\circ\lambda_{\sigma\tau}^{-1}$. 
  \item The cohomology class $\gamma_B\in H^2(G_k,R^*)$ only depends on the $\Qbar$-isogeny class of $B$. More precisely, if $\beta\colon A \ra B$ is an isogeny, then $\gamma_A = \gamma_B$, under the identification of centers $Z(\cB)\simeq Z(\End(A))$ provided by $\beta$.
\item For every $n\in \Z_{\geq 1}$ the isotypical variety $B^n$ is also a $k$-variety, and $\gamma_B=\gamma_{B^n}$, under the identification $Z(\cB)\simeq Z(\cB^n)$.
  \end{enumerate}
\end{proposition}
\begin{proof}
  For $i)$, the cocycle $c_B'$ can be written as $c_B'(\sigma,\tau)=c_B(\sigma,\tau)\circ a_\sigma\circ a_\tau\circ a_{\sigma\tau}^{-1}$ for some $a_\sigma,a_\tau,a_{\sigma\tau}\in R^*$. Then, for every $s\in G_k$ the isogeny $\lambda_s = a_s \circ \mu_s$ is compatible, because $a_s\in R^*$, and it satisfies the required property.

As for $ii)$ one can define isogenies $\lambda_\sigma\colon \acc \sigma A \ra A$ as $\lambda_\sigma=\beta^{-1}\circ \mu_\sigma \circ \acc\sigma \beta$. It is easy to check that they are compatible and that $c_A(\sigma,\tau) = \beta^{-1}\circ c_B(\sigma,\tau)\circ\beta$.

Finally, $iii)$ is proved by considering the isogenies $\mu_\sigma^{\oplus n}\colon \acc\sigma B^n\ra B$, which are compatible and can be thus used to compute $c_{B^n}(\sigma,\tau)$, which is seen to coincide with $c_B^{n\oplus}$.
\end{proof}
\begin{proposition}\label{proposition: nttors}
  Let $B/\Qbar$ be a simple abelian $k$-variety and suppose that $\End(B)$ has Schur index $t$. Suppose that there exists a variety $A/k$ such that $A_\Qbar\sim B^n$. Then $\gamma_B$ lies in $H^2(G_k,R^*)[nt]$; that is, $(\gamma_B)^{nt}=1$. 
\end{proposition}
\begin{proof}
 Let $\{\mu_\sigma\colon \acc\sigma B \ra B\}_{\sigma\in G_k}$ be a compatible system of isogenies and let $c_B(\sigma,\tau)=\mu_\sigma\circ \acc\sigma\mu_\tau\circ\mu_{\sigma\tau}^{-1}$ be the corresponding cocycle. By Proposition \ref{prop: main props k-variety} $ii)$ we have that $A_\Qbar$ is a $k$-variety as well, and therefore there exists a compatible system of isogenies $\alpha_\sigma\colon \acc\sigma A \ra A$ such that 
 \begin{align}\label{eq: alpha_A}
c_A(\sigma,\tau)=\alpha_\sigma\circ \acc\sigma\alpha_\tau\circ\alpha_{\sigma\tau}^{-1}.
 \end{align}
Since $A$ is defined over $k$, we see that $\alpha_s$ lies in $\End(A_\Qbar)$ for every $s\in G_k$. In particular, the compatibility condition applied to the endomorphism $\alpha_\tau$ then reads $\acc\sigma\alpha_\tau= \alpha_\sigma^{-1}\circ \alpha_\tau\circ \alpha_\sigma$. Plugging this into \eqref{eq: alpha_A} gives
\begin{align*}
  c_A(\sigma,\tau)=\alpha_\tau\circ\alpha_\sigma\circ\alpha_{\sigma\tau}^{-1}.
\end{align*}
Now by Proposition \ref{prop: main props k-variety} $ii)$ we have that $\gamma_A=\gamma_{B^n}$; by  $iii)$ we also see that $\gamma_A=\gamma_B$ and by $i)$ we can in fact assume that the cocycles are equal:
\begin{align}\label{eq: cA equals cB}
  c_B(\sigma,\tau)=c_A(\sigma,\tau)=\alpha_\tau\circ\alpha_\sigma\circ\alpha_{\sigma\tau}^{-1},
\end{align}
where the term $c_B(\sigma,\tau)$ is seen as lying in $\End(A_\Qbar)$ by means of the identification $R\simeq Z(\End(A_\Qbar))$. In turn, we can interpret \eqref{eq: cA equals cB} as an equality in $\M_n(\cB)$ thanks to the isomorphism $\End(A_\Qbar)\simeq \M_n(\cB)$. Now we take reduced norms and we use the fact that $c_B(\sigma,\tau)$ lies in $F$, so that its reduced norm equals $c_B(\sigma,\tau)^{nt}$. Thus we obtain:
\begin{align*}
  c_B(\sigma,\tau)^{nt} = \operatorname{nr}(\alpha_\tau)\circ\operatorname{nr}(\alpha_\sigma)\circ \operatorname{nr}(\alpha_{\sigma\tau})^{-1}.
\end{align*}
This expresses $c_B^{nt}$ as the coboundary of the map $s\mapsto \operatorname{nr}(\alpha_s)\in R^*$.
\end{proof}
\begin{corollary}\label{cor: cocycle of order g}
  Suppose that $A/k$ is an abelian variety such that $A_\Qbar\sim C^g$, for some $k$-curve $C/\Qbar$. Then $\gamma_C$ lies in $ H^2(G_k,R^*)[g]$. 
\end{corollary}
\begin{proof}
  It follows directly from the fact that $\End(C)$ has Schur index $1$.
\end{proof}
\begin{remark}
  Ribet showed \cite[Prop. 3.2]{Rib94} that if $C$ is a $k$-curve without CM then $\gamma_C$ has order dividing $2$. This is not true in general for CM $k$-curves. The point of Corollary \ref{cor: cocycle of order g} is that an analogous statement holds in the case that $C$ has CM and $C^g$ admits a model up to isogeny defined over $k$.
\end{remark}

\begin{proposition}\label{prop: center}
  Let $L\subset \Qbar$ be an extension of $k$ and let  $A/L$ be a $k$-variety. Then $Z(\End(A_\Qbar))\subseteq \End(A)$. 
\end{proposition}
\begin{proof}
  For all $\sigma\in G_L$ we have that $\acc\sigma A = A$. Then the isogeny $\mu_\sigma\colon \acc\sigma A \ra A$ can be identified with an element of $\End(A)$. For any $\varphi\in Z(\End(A_\Qbar))$ we have that $\acc\sigma \varphi = \mu_\sigma^{-1}\circ \varphi \circ\mu_\sigma =\varphi$.
\end{proof}

The main examples of $k$-varieties that we will consider in this note are CM elliptic curves whose field of CM is contained in $k$. The following result is well-known (it follows, for instance, from \cite[Thm. 2.2]{Sil94}).
\begin{proposition}\label{proposition: CM kcurve}
Let $E/L$ be an elliptic curve with CM by an imaginary quadratic field $M$.  If $L$ contains $M$, then $E$ is an $M$-curve. In particular, it is a $k$-curve for any $k$ containing $M$.
\end{proposition}

For further reference, we record the following consequence of the results we have seen so far.
\begin{proposition}\label{proposition: centre racional}
  Let $A$ be an abelian surface defined over a number field $k$. Suppose that $A_\Qbar\sim E^2$, where $E$ is an elliptic curve with CM by $M$, and assume that $M\subseteq k$. Then $Z(\End(A_\Qbar))\subseteq \End(A)$.
\end{proposition}
\begin{proof}
  Since $E$ is a CM elliptic curve defined over an extension of $k$ and $M\subseteq k$ we see that $E$ is a $k$-curve. Now by Proposition \ref{prop: main props k-variety} $A$ is also a $k$-variety, and the result is then a direct consequence of Proposition \ref{prop: center}.\end{proof}
We will use the notion of $k$-varieties completely defined over a field, a terminology which was introduced in \cite[p. 2]{Qu00}.
\begin{definition}\label{def: completely defined}
Let $B$ be a $k$-variety defined over a number field $K$. One says that~$B$ is completely defined over $K$ if there exist compatible isogenies $\{\mu_\sigma\}_{\sigma\in G_k}$ that are defined over $K$.  
\end{definition}
If $B$ is completely defined over $K$ the map 
\begin{align}\label{eq: completely defined}
c_B^K\colon \Gal(K/k)\times \Gal(K/k) \lra    R^*, \ \ \ c_B^K(\sigma,\tau) = \mu_\sigma\circ \acc\sigma \mu_\tau\circ\mu_{\sigma\tau}^{-1}
\end{align}
is a two-cocycle and its cohomology class $\gamma_B^K\in H^2(\Gal(K/k),R^*)$ only depends on the $K$-isogeny class of $B$.
\begin{remark}\label{remark: imageinf} It follows from the definitions that the image of $\gamma_B^K$ under the inflation map $H^2(\Gal(K/k),R^*)\ra H^2(G_k,R^*)$ is precisely the cohomology class $\gamma_B$ defined earlier. In many applications, for instance when one is interested in studying properties of the $K$-isogeny class of $B$, it is important to work with the cohomology class $\gamma_B^K$.
\end{remark}

\subsection{Powers of CM elliptic $k$-curves}\label{section: field def}

Throughout this section, let $E$ be an elliptic curve defined over $\Qbar$ with CM by a quadratic imaginary field $M$ such that~$E^g$ is $\Qbar$-isogenous to an abelian variety~$A$ defined over~$k$. Assume that $M$ is contained in~$k$. Note that, in view of Proposition~\ref{proposition: CM kcurve}, $E$ is a $k$-curve. 

Let $K/k$ denote the smallest extension such that $\End(A_K)=\End(A_\Qbar)$.
The main result of this section is Theorem~\ref{theorem: ggext}, which ensures the existence of an abelian subextension of $K/k$ with Galois group killed by~$g$ over which $E$ admits a model up to isogeny.

\begin{remark}\label{remark: rib} Ribet \cite{Rib94} and Elkies \cite{Elk04} have given two alternative proofs of the fact that any $k$-curve without CM admits a model up to isogeny defined over a polyquadratic extension of $k$. Theorem~\ref{theorem: ggext} may be seen as a extension of Ribet's and Elkies' result to the case  in which $E$ has CM, but under the restrictive assumption that~$E^g$ admits a model up to isogeny over~$k$.
\end{remark}

Before proceeding to the proof of Theorem~\ref{theorem: ggext}, we introduce some notation and we recall Weil's descent criterion, a fundamental result in our arguments.

Note that the definition of $K/k$ implies that $A_K\sim E_1^g$, where $E_1$ is an elliptic curve defined over $K$. Since
$$
\acc\sigma E_1^g\sim \acc \sigma A_K\sim A_K\sim E_1^g\,,
$$
for $\sigma \in G_k$, by Poincar\'e's decomposition theorem there exists an isogeny
$$
\mu_\sigma\colon \acc\sigma E_1\rightarrow E_1
$$
defined over $K$. Note that $E_1$ is a model of $E$ up to isogeny defined over $K$. Even more, according to the terminology of Definition~\ref{def: completely defined}, $E_1$ is an elliptic $k$-curve completely defined over $K$.
Define the 2-cocycle
$$
c_{E_1}\colon G_k\times G_k\rightarrow M^*\,,\qquad c_{E_1}(\sigma,\tau)=\mu_\sigma \circ\acc{\sigma} \mu_{\tau}\circ\mu_{\sigma\tau}^{-1},
$$
and let $\gamma = \gamma_{E_1}$ denote its cohomology class in $H^2(G_k,M^*)$. By Remark~\ref{remark: imageinf}, the class $\gamma$ lies in the image of the inflation map 
\begin{align*}
\operatorname{Inf}\colon H^2(\Gal(K/k),M^*)\lra H^2(G_k, M^*).
\end{align*} 
The following result \cite[Thm. 8.2]{Rib92}, written in the spirit of \cite[Prop. 3.1]{Rib94}, is crucial for our purposes.

\begin{proposition}[Weil's descent criterion]\label{proposition: Weil} Let $F/k$ be an algebraic extension. If $\gamma$ lies in the kernel of the restriction map 
$$
H^2(G_k,M^*)\rightarrow H^2(G_F,M^*)\,,
$$ 
then $E_1$ admits a model up to isogeny defined over $F$. 
\end{proposition}
\begin{proof}
Of course it is enough to prove that $\gamma=1$, implies that $E_1$ admits a model up to isogeny defined over $k$. But $\gamma=1$ means that there is a locally constant function $d\colon G_k\rightarrow M^*$ such that $c_{E_1}(\sigma\tau)=d(\sigma)d(\tau)/d(\sigma\tau)$. Then the isogenies $\lambda_{\sigma}:=1/d(\sigma)\circ \mu_{\sigma}$ satisfy $\lambda_\sigma \circ\acc{\sigma} \lambda_{\tau}\circ\lambda_{\sigma\tau}^{-1}=1$, and then one concludes by applying \cite[Thm. 8.2]{Rib92}.
\end{proof}
\begin{remark}\label{rk: K-isogenous to a curve over F}
  Let $\gamma^K\in H^2(\Gal(K/k),M^*)$ be the cohomology class constructed as in \eqref{eq: completely defined}. Since $\gamma^K$ is constructed in terms of isogenies that are defined over $K$, the argument above can be adapted to show that if $F/k$ is subextension of $K/k$ and $\gamma^K$ lies in the kernel of the restriction map 
  \begin{align*}
    H^2(\Gal(K/k),M^*)\lra H^2(\Gal(K/F),M^*),
  \end{align*}
then $E_1$ is $K$-isogenous to an elliptic curve defined over $F$.
\end{remark}
\begin{theorem}\label{theorem: ggext}  Let~$E$ be an elliptic curve defined over $\Qbar$ with CM by a quadratic imaginary field~$M$ such that~$E^g$, for $g\geq 1$, is $\Qbar$-isogenous to an abelian variety~$A$ defined over~$k$. Assume that $M$ is contained in~$k$. Let $H$ denote the Hilbert class field of $M$ and write $F=Hk$. Then, $E$ admits a model up to isogeny defined over the abelian subextension $F/k$ of $K/k$ and every element of $\Gal(F/k)$ has order dividing $g$.
\end{theorem}

\begin{proof} Of course, it is enough to prove the equivalent statement for $E_1$ (where~$E_1$ is attached to $E$ as above). We will follow the strategy of \cite[Thm. 3.3]{Rib94}. Although, Ribet's result is stated under different hypotheses, we will show that~$\gamma=\gamma_{E_1}$ lying in $H^2(G_k,M^*)[g]$ is essentially all that is needed to run his argument. Let~$U$ denote the group of roots of unity\footnote{More explicitly, $U=\{\pm 1\}$, unless $E_1$ has CM by $M=\Q(\sqrt{-1})$ or $\Q(\sqrt{-3})$, in which case $U=\langle \zeta_4\rangle$ or $\langle \zeta_6\rangle$, respectively.} of $M$ and write $P:=M^*/U$. As in \cite[Lemma 3.5]{Rib94}, it follows from Dirichlet's unit theorem that $P$ is a free abelian group. Therefore, the exact sequence
$$
1\rightarrow U \rightarrow M^*\rightarrow P\rightarrow 1
$$
is split. This yields a decomposition
\begin{equation}\label{equation: dec}
H^2(G_k,M^*)[g]  \simeq  H^2(G_k, U)[g]\times \Hom(G_k,P/P ^g)\,.
\end{equation}
Indeed, the long exact cohomology sequence associated to the exact sequence
$$
1\rightarrow P\stackrel{x\mapsto x^g}\longrightarrow P\rightarrow P/P^g\rightarrow 1\,,
$$
together with the freeness of $P$, implies that we have an isomorphism 
$$
H^2(G_k,P)[g]\simeq \Hom(G_k,P/P^g)\,.
$$
Note that by Corollary~\ref{cor: cocycle of order g}, the class~$\gamma$ belongs to $H^2(G_k,M^*)[g]$. Let~$\gamma_{U}$ (resp. $\overline\gamma$) be the projection of $\gamma$ onto the first (resp. second) factor of the decomposition~\eqref{equation: dec}. Since~$\gamma$ is the inflation of a class in $H^2(\Gal(K/k), M^*)$, we see that the component~$\overline\gamma$ lies in the image of the inflation map \begin{align*}
\mathrm{Inf}\colon \Hom(\Gal(K/k), P/P^g)\lra \Hom(G_k,P/P^g).
\end{align*}
In particular, if we let $L$ be the subfield of $K$ fixed by the subgroup of $\Gal(K/k)$ generated by the $g$-th powers of the elements of $\Gal(K/k)$, it is clear that the restriction of $\overline\gamma$ to $G_L$ is trivial. But now, by Lemma~\ref{lemma: gammabartrivializes} below, $\gamma$ is trivial over~$G_L$ and thus, by Proposition~\ref{proposition: Weil}, $E_1$ admits a model~$E_2$ over~$L$ up to isogeny. Let~$\mathcal O_c$ be the order, say of conductor $c$, by which $E_2$ has CM, and let $H_c=M(j(E_2))$ denote the ray class field of conductor $c$ of $M$. Since $E_2$ is defined over $L$, we have
\begin{equation}\label{equation: inclus1 class}
kH_c=k(j(E_2))\subseteq L. 
\end{equation}
Let $\mathcal O$ be the maximal order of $M$ and let $E^*$ be an elliptic curve with CM by~$\mathcal O$ defined over $F=kH$. Note that $E^*$ is $\Qbar$-isogenous to $E_2$. We have
\begin{equation}\label{equation: inclus2 class}
F=k(j(E^*))=kH\subseteq kH_c. 
\end{equation}
The theorem now follows from \eqref{equation: inclus1 class}, \eqref{equation: inclus2 class}, and the fact that $\Gal(L/k)$ is killed by~$g$.
\end{proof}

\begin{lemma}\label{lemma: gammabartrivializes}
With the notations as in the proof of Theorem~\ref{theorem: ggext}, if $\overline \gamma$ trivializes over a subextension $L/k$ of $K/k$, then so does $\gamma_{U}$. That is, if $\overline \gamma$ has trivial restriction to $H^2(G_L,M^*)$, then so does $\gamma_{U}$.
\end{lemma}

\begin{proof} It is enough to prove that if $\overline{\gamma}= 1 $ then $\gamma_U=1$.  To show this, we will follow Ribet again. As proven in \cite[\S 4]{Rib94}, we have
$$
H^1(G_k,(M\otimes_\Q \Qbar)^*)=1, \quad H^2(G_k,(M\otimes_\Q \Qbar)^*)\simeq H^2(G_k,\Qbar^*)\oplus H^2(G_k,\Qbar^*)\,,
$$
where $G_k$ acts trivially on $M$ and via the natural Galois action on $\Qbar$.
 
Then, the exact sequence
$$
1 \rightarrow M^*\ra (M\otimes_\Q \Qbar)^*\ra  (M\otimes_\Q \Qbar)^*/M^*\ra 1,
$$
yields the following long exact sequence in cohomology 
\small
$$
1 \ra H^1(G_k, (M\otimes_\Q \Qbar)^*/M^*)\stackrel{\delta}{\ra} H^2(G_k,M^*)\ra H^2(G_k,\Qbar^*)\oplus H^2(G_k,\Qbar^*)\ra\cdots.
$$\normalsize
As in \cite[Lem. 4.3]{Rib94}, one finds that $\gamma$ lies in the image of $\delta$. Indeed, in our case $M\otimes \Qbar\simeq \Qbar\oplus\Qbar$ as $G_k$-modules. Fix an invariant differential $\omega$ of $E_1$. For every $\sigma\in G_k$ there is a unique $\lambda_\sigma\in\Qbar$ characterized by the identity $\mu_\sigma^*(\omega) = \lambda_\sigma \cdot \acc\sigma\omega$. Then the map $\sigma\mapsto (\lambda_\sigma,\overline\lambda_\sigma)$, where the bar stands for complex conjugation, gives rise to a cohomology class in $H^1(G_k,(M\otimes\Qbar)^*/M)$, whose image under $\delta$ coincides with $\gamma$.

Since $\overline \gamma$ is trivial, this means that $\gamma_{U}$ lies in the kernel of the map 
$$
j\colon H^2(G_k, U)[g]\ra  H^2(G_k,\Qbar^*)\oplus H^2(G_k,\Qbar^*).
$$
Observe that $U$ is contained in $k$ so that the trivial action of $G_k$ on $U$ coincides with the natural Galois action. Since $H^2(G_k, U)\simeq H^2(G_k,\Qbar^*)[n]$ (here $\Qbar$ is endowed with the natural Galois action of $G_k$, and $n$ is the cardinality of $U$), we see that $j$ is injective and thus $\gamma_{U}$ must be trivial.
\end{proof}

\begin{corollary}\label{corollary: class group}
Let $A$ be an abelian variety defined over $\Q$ of dimension $g$ that is $\Qbar$-isogenous to $E^g$, where $E$ is an elliptic curve over $\Qbar$ with CM by $M$. Then every element in $\Cl(M)$ has order dividing $g$.
\end{corollary}

\begin{proof}
Applying Theorem~\ref{theorem: ggext} to the base change $A_M$ of $A$ to $M$, we deduce that every element in $\Gal(H/M)$ has order dividing $g$, where $H$ stands for the Hilbert class field of $M$. The corollary follows from the fact that $\Gal(H/M)$ is isomorphic to $\Cl(M)$.
\end{proof}

\begin{corollary}\label{corollary: classbiquad_new} Let $A$ be an abelian variety defined over $\Q$
that is $\Qbar$-isogenous to~$E^g$, where $E$ is an elliptic curve
defined over $\Qbar$ with CM by a quadratic imaginary field $M$ and
$g$ is a prime number. Then 
$$
\Cl(M)=
\begin{cases}
1,\cyc g,\text{ or }\cyc g \times \cyc g & \text{ if } g=2,\\ 
1\text{ or } \cyc g & \text{ otherwise.}
\end{cases}
$$
In particular, if $g=2$, there are at most $51$ possibilities for $M$. 
\end{corollary}

\begin{proof}
 By Corollary \ref{corollary: class group}, the group $\Cl(M)$ is a quotient of $\Gal(K/M)$ of the form $\cyc g^r$, for some $r\geq 0$. For $g=2$, the possibilities for $\Gal(K/M)$ are well known (see Remark~\ref{remark: r(A)} in the next section, for example) and it turns out that $r\leq 2$. It is well known that the number of quadratic imaginary fields with class group isomorphic to the trivial group, $\cyc 2$, $\cyc 2 \times  \cyc 2$ are respectively $9$, $18$, $24$ (see Table~\ref{table: Ms} for the complete list).

Suppose now that $g\geq 3$. For $p$ a prime number, define
\begin{align}\label{eq: sil}
r(g,p):=\sum_{j\geq 0} \left\lfloor\frac{2g}{p^j(p-1)}\right\rfloor.
\end{align}
By \cite[Thm. 4.1]{Sil92}, the maximal power of $p$ dividing the order of $\Gal(K/M)$ is $r(g,p)$. Guralnick and Kedlaya \cite[Thm. 1.1]{GK16} have shown that the same holds true if one replaces $r(g,p)$ by
$$
r'(g,p):=\begin{cases}
r(g,p)-g-1 & \text{if $p=2$,}\\
\max\{r(g,p)-1,0\} & \text{if $p$ is a Fermat prime,}\\
r(g,p) & \text{otherwise.}\\
\end{cases}
$$  
This bound is sharp in the context of general abelian varieties (see \cite[Thm. 6.3]{GK16}). However, if~$A$ is as in the statement of the corollary (that is, geometrically isogenous to the power of an elliptic curve with CM), one can do better. Indeed, one can specialize the proof of \cite[Thm. 5.4]{GK16} to the particular case in which $V\simeq W^g$ as $G^0$-modules, as in our setting $A_K\sim E_1^g$ for an elliptic curve $E_1$ defined over $K$. Therefore $\End_{G^0}(W)=M$, $\End_{G^0}(V)=\M_g(M)$, and $Z(\End_{G^0}(W))=M$. In the notations of \cite[Thm. 5.4]{GK16}, this is yields the values
$$
a=2, \qquad b =1, \qquad c=g\,.
$$
Then, one deduces that the maximal power of $g$ dividing the order of $\Gal(K/M)$ is bounded by
$$
r(\GL(bc)_M,g)=r(\GL(g)_M,g):=m(M,g)\left\lfloor{\frac{g}{t(M,g)}}\right\rfloor+\left\lfloor{\frac{1}{t(M,g)}}\right\rfloor\,,
$$
where, as defined in \cite[Rmk. 3.6]{GK16}, we have 
$$
m(M,g):=\min \{ m\geq 1\, | \, M\cap \Q(\mu_{g^m})=M\cap \Q(\mu_{g^\infty})\}
$$
and 
$$
t(M,g):=[\Q(\mu_{g^{m(M,g)}}):M\cap \Q(\mu_{g^{m(M,g)}})]\,.
$$
Since $\Q(\mu_{g^{m(M,g)}})$ ramifies only at $g$, we have that either $t(M,g)=g^{m(M,g)-1}(g-1)$ or $M=\Q(\sqrt{-g})$ and $g\equiv 3 \pmod 4$. In the first case, one trivially checks that $r(\GL(g)_M,g)\leq 1$ (here we use $g\not =2$). In the second case, $\Cl(M)$ is necessarily trivial; indeed, by the reduction theory of positive definite binary quadratic forms (see for example \cite[Chap. I, \S2]{Cox89}) one knows that $\Cl(M)\leq \frac{2}{3} g$. 
\end{proof}

\begin{remark}
There is a finite number of quadratic imaginary fields of fixed class number (see \cite{Hei34}). Using Sage and Pari one can compute the number of imaginary quadratic fields with class group $\cyc g$, for $3\leq g\leq 97$ prime, relying on the bounds for their fundamental discriminants provided by \cite[Table 4]{Wat03}.
\end{remark}

\begin{remark}
Also for $g$ not necessarily prime, one can proceed as in the proof of Corollary~\ref{corollary: classbiquad_new} to gain control on the size of $\Cl(M)$. In general, one obtains that $\Cl(M)$ has exponent dividing $g$ and order dividing $\prod_{p|g} p^{r'(g,p)}$.
\end{remark}

\begin{remark}
One may wonder which of the cases allowed by Corollary~\ref{corollary: classbiquad_new} actually arise. By taking $g$-th powers (resp. Weil restrictions of scalars) of elliptic curves it is clear that all quadratic imaginary fields of class number 1 (resp. with class group~$\cyc g$) do appear. It would be an interesting problem to determine the exact set of possibilities for the $M$ with $\Cl(M)\simeq \cyc 2 \times \cyc 2$ for which an abelian surface~$A$ as in the corollary exists.
\end{remark}

The above corollary answers part $(A)$ of Question~\ref{question: sobre M}. Providing an answer to part $(B)$ of the question for $g=2$ will be the goal of \S\ref{section: dim2}. Before, for the sake of completeness, we state a result analogous to Theorem~\ref{theorem: ggext} for the non-CM case.

\begin{theorem}\label{theorem: nonCMdef}
Let $E$ be an elliptic curve defined over $\Qbar$ without CM such that~$E^g$ is $\Qbar$-isogenous to an abelian variety $A$ defined over~$k$. Then $E$ admits a model up to isogeny defined  
\begin{enumerate}[i)]
\item over $k$ if $g$ is odd; or
\item over a polyquadratic subextension of $K/k$ if $g$ is even. 
\end{enumerate}
\end{theorem}

\begin{proof}
Let $\gamma$ denote the cohomology class attached to $E$. On the one hand, by Proposition~\ref{proposition: nttors}, the class $\gamma^g$ is trivial. On the other hand, \cite[Prop. 3.2]{Rib94} asserts that $\gamma$ has order dividing $2$.  Proposition~\ref{proposition: Weil} implies then that $i)$ holds. It is straightforward to adapt the proof of Theorem~\ref{theorem: ggext} to justify that the polyquadratic extension over which $\gamma$ trivializes is in fact contained in $K/k$.
\end{proof}

\section{The case of dimension $g=2$}\label{section: dim2}

In this section, $E$ will denote an elliptic curve defined over $\Qbar$ with CM by a quadratic imaginary field $M$ such that~$E^2$ is $\Qbar$-isogenous to an abelian surface~$A$ defined over~$k$. We assume that $M$ is contained in~$k$, so that $E$ is a $k$-curve, and we let $K/k$ denote the smallest extension such that $\End(A_K)=\End(A_\Qbar)$. 

Our goal is to obtain a more precise version of Corollary~\ref{corollary: classbiquad_new}. This will be achieved by Theorem~\ref{theorem: absurfQ}. We will benefit from the fact that for $g=2$, the structure of $\Gal(K/k)$ is very well understood.

\begin{remark}\label{remark: r(A)}
It is well known that if $M\subseteq k$ the possibilities for $\Gal(K/k)$ are the cyclic group $\cyc n$ for $n\in \{1,2,3,4,6\}$, the dihedral group~$\dih n$ for $n\in \{2,3,4,6\}$, the alternating group $\alt 4$, and the symmetric group $\sym 4$.
There are at least two ways in which this can be deduced: first, by means of the analysis of the finite subgroups of $\PGL_2(M)$, (see \cite{Bea10}, \cite{CF00}); second, as a byproduct of the classification of Sato--Tate groups of abelian surfaces (see \cite[\S4]{FKRS12}). 
\end{remark}

Let $r(A)$ be defined in terms of $\Gal(K/k)$ as specified on Table~\ref{table: rA}. In other words, $r(A)$ is the maximum value of $r$ for which there exists a Galois subextension $F/k$ of $K/k$ with $\Gal(F/k)\simeq \cyc 2^r$.

\begin{table}
\begin{center}
\footnotesize
\setlength{\extrarowheight}{0.5pt}
\begin{tabular}{ccc|ccc}
\hline
$r(A)$ & $\Gal(K/k)$ & \qquad &  \quad  & $r(A)$ & $\Gal(K/k)$\\
\hline 
 0 & $\cyc 1$ & &  & 2 & $\dih 2$ \\
 1 & $\cyc 2$ & &  & 1 & $\dih 3$  \\
 0 & $\cyc 3$ & &  & 2 & $\dih 4$ \\
 1 & $\cyc 4$ & &  & 2 & $\dih 6$ \\
 1 & $\cyc 6$ & &  & 0 & $\alt 4$ \\
   & & &  & 1 & $\sym 4$ \\
\hline
\end{tabular}
\vspace{6pt}
\caption{The rank $r(A)$ in terms of $\Gal(K/k)$.}\label{table: rA}
\end{center}
\end{table}

\begin{proposition}\label{proposition: biquad}
Let $E$ be an elliptic curve defined over $\Qbar$ with CM by a quadratic imaginary field $M$ such that $E^2$ is $\Qbar$-isogenous to an abelian surface $A/k$. Assume that $M\subseteq k$. Then $E$ admits a model up to isogeny $E^*$ defined over a Galois subextension $F/k$ of $K/k$ such that $\Gal(F/k)\simeq \cyc 2^{r(A)}$.
\end{proposition}

The above proposition can certainly be deduced from Theorem~\ref{theorem: ggext} and Remark~\ref{remark: r(A)}, but we wish to present a slightly different argument that renders, for $g=2$, a shortcut in the proof.

\begin{proof}
Recover decomposition \eqref{equation: dec}, specialize it to $g=2$, and resume with the final argument in the proof of \cite[Thm. 3.3]{Rib94}. If $\Br(k)$ denotes the Brauer group of $k$, then we have
$$
H^2(G_k,M^*)[2]\simeq \Br(k)[2]\times \Hom(G_k,P/P^2)\,.
$$
If we let $\gamma=\gamma_{E}$ denote the cohomology class attached to $E$ and $(\gamma_U, \overline\gamma)$ its components under the above decomposition, then $\overline\gamma$ factors through a polyquadratic extension $k'$ of $k$. Moreover, a theorem of Merkur'ev \cite{Mer81} shows that $\Br(k)[2]$ is generated by the classes of quaternion algebras over~$k$, and it is well known that each quaternion algebra over~$k$ is split over a quadratic extension of~$k$. Proposition~\ref{proposition: Weil} thus tells us that $E$ admits a model $E_2$ up to isogeny over a polyquadratic extension $k_2/k$.

Now let $E_1/K$ be an elliptic curve such that $A_K\sim E_1^2$. We thus have that
\begin{equation}\label{equation: fieldinc}
\Q(j(E_1))\subseteq K\quad\text{and}\quad\Q(j(E_2))\subseteq k_2\,.
\end{equation}  
We can now conclude in a similar manner as we did with the proof of Theorem~\ref{theorem: ggext}. Let $\mathcal O$ be the maximal order of $M$ and let $E^*$ be an elliptic curve with CM by~$\mathcal O$. Both $E_1$ and $E_2$ are $\Qbar$-isogenous to $E^*$. Let $\mathcal O_i$ denote the order by which $E_i$ has CM, for $i=1, 2$, and let $c_i$ be its conductor. Since the ray class field $H_{c_i}=M(j(E_i))$ of conductor $c_i$ contains the Hilbert class field $H=M(j(E^*))$, we have that $M(j(E^*))\subseteq M(j(E_i))$. It follows from (\ref{equation: fieldinc}) that 
$$
F:=k(j(E^*))\subseteq K\cap k_2\,.
$$
Since $[K\cap k_2: k]\leq 2^{r(A)}$, we have that $E^*_F$ is a model up to isogeny of $E$ enjoying the properties stated in the proposition.
\end{proof}

\begin{remark}\label{rk: field of def} It follows from the above proof that any elliptic curve with CM by the maximal order of $M$ can be taken as the model $E^*$ of Proposition~\ref{proposition: biquad}, in which case $F$ can be taken equal to $k(j(E^*))$. 
\end{remark}

\begin{remark}
If $E$ does not have CM, then the possibilities for $\Gal(K/k)$ are those of Remark~\ref{remark: r(A)} excluding $\alt 4$ and $\sym 4$. Theorem~\ref{theorem: nonCMdef} has then the following consequence. Let $E$ be an elliptic curve without CM defined over $\Qbar$ such that~$E^2$ has a model up to isogeny defined over~$k$. Then~$E$ admits a model up to isogeny~$E^*$ defined over a biquadratic subextension $F/k$ of $K/k$.
\end{remark}

In the following \S\ref{section: galstr} and \S\ref{sec: restriction of scalars}, we will gather constraints on the field $M$ imposed by $\Gal(K/M)$. In \S\ref{section: imquad}, we will achieve a refinement of Corollary~\ref{corollary: classbiquad_new} by putting together Proposition~\ref{proposition: biquad} and all of these constraints.

\subsection{Galois module structures: group representation obstructions}\label{section: galstr} We continue with the same notation for $A$, $E$, $K/k$, and $M$. By Proposition~\ref{proposition: biquad},~$E$ admits a model up to isogeny $E^*$ defined over a subextension $F/k$ of $K/k$. We will consider the following hypotheses: 

\begin{itemize}
\item[$(H_4)$] $\Gal(K/F)$ contains an element $\tau$ of order $4$.
\item[$(H_6)$] $\Gal(K/F)$ contains an element $\tau$ of order $6$.
\item[$(\alt 4)$] $\Gal(K/F)$ contains $\alt 4$.
\item[$(\sym 4)$] $\Gal(K/F)$ contains $\sym 4$ (equivalently, $\Gal(K/F)\simeq \sym 4$).
\end{itemize}

Denote by $L/F$ the minimal extension over which $\Hom(E_{L}^*,A_L)\simeq \Hom(E_{\Qbar}^*,A_{\Qbar})$. Observe that~$K$ is contained in $L$.
The main result of this section is the following proposition, which imposes conditions on $M$ in terms of $\Gal(K/F)$ and $\Gal(L/F)$.
\begin{proposition}\label{proposition: M} One has:
\begin{enumerate}[i)]
\item If $(H_4)$ holds, then $M=\Q(\sqrt{-1})$ or $M=\Q(\sqrt{-2})$.
\item If $(H_6)$ holds, then $M=\Q(\sqrt{-3})$.
\item If $(\alt 4)$ holds, then $M$ splits the quaternion algebra $(-1,-1)_\Q$. \footnote{Recall that $M=\Q(\sqrt{-d})$, with $d$ square free, splits $(-1,-1)_\Q$ if and only if $d\not\equiv 7 \mod 8$.}
\item If $(\sym 4)$ holds, then:
\begin{itemize}
\item $M=\Q(\sqrt{-2})$ if and only if $\Gal(L/F)$ is isomorphic to the binary octahedral group $\mathrm{B_O}$ or the group $\GL_2(\Z/3\Z)$. 
\item $M=\Q(\sqrt{-1})$ if and only if $\Gal(L/F)$ is isomorphic to the complex reflexion group $G_8$.
\end{itemize}
\end{enumerate}
\end{proposition}

The proof is built on some results of \cite{FS14}, which we now recall. We note also that the proofs of statements $i)$ and $ii)$, and of $iii)$ and $iv)$ require slightly different techniques, so we will treat them separately. First, let us introduce some notations and make some basic observations. 

\begin{lemma}\label{lemma: fieldL}
One has that $\Gal(L/K)\subseteq\cyc n$, where $n$ is the number of roots of unity in $M$. 
\end{lemma}

\begin{proof}
First observe that $L/F$ is a Galois extension by the minimality condition defining it. Let~$E_1$ be an elliptic curve defined over~$K$ such that $A_K\sim E_1^2$. Then $L/F$ may be characterized by the fact that $L/K$ is the minimal extension over which $E_1$ and $E_K^*$ become isogenous. Let $\psi\colon E_{1,\Qbar}\rightarrow E^*_\Qbar$ be an isogeny. Since $G_K$ acts trivially on $M=\End(E_1)$, the map
$$
\xi\colon G_K\rightarrow M^*\\,\qquad \xi(\sigma):=\psi^{-1}\circ \acc \sigma \psi
$$
is a homomorphism. But then
$$
\Gal(L/K)\simeq G_K/\ker(\xi)\simeq \im(\xi)\subseteq \cyc n\,,
$$
as desired.
\end{proof}

Observe that the action of $M\simeq Z(\End(A_\Qbar))$ on $\Hom(E_{L}^*,A_L)$ commutes with that of~$G_F$, since $Z(\End(A_\Qbar))\subseteq \End(A)$ by Proposition~\ref{proposition: centre racional}.
By letting~$G_{F}$ act naturally on the first factor and trivially on~$M$, the tensor product
$$
\Hom(E_{L}^*,A_L)\otimes_{M,\iota}M\qquad (\text{resp.}\quad\End(A_L)\otimes_{M,\iota}M)\,,
$$
taken with respect to any of the two automorphisms $\iota\in \{\id,c\}$ of $M$, becomes a $M[\Gal(L/F)]$-module of dimension $2$ (resp. $4$). Here, $c$ denotes the non-trivial automorphism\footnote{To ease notation, for $a \in M$, we will simply write $\overline a$ to denote $\acc c a$.} of $M$.
Let $\theta_{\iota}(E^*)$ (resp. $\theta_{\iota}(A)$) denote the representation afforded by this module. 
 
Let $\pi\colon\Gal(L/F)\rightarrow \Gal(K/F)$ be the natural projection.
Note that $\theta_\iota(E^*)$ is (by the definition of $L/F$) a faithful representation of $\Gal(L/F)$, whereas $\theta_\iota(A)$ is only faithful (by the definition of $K/k$) as a representation of $\Gal(K/F)$. The following lemma is a restatement of \cite[Prop. 3.2]{FS14} and \cite[Prop. 3.4]{FS14} in our setting.

\begin{lemma}\label{lemma: FS14} One has:
\begin{enumerate}[i)]
\item $\Tr(\theta_\iota(A))=\Tr(\theta_{\id}(E^*))\cdot
\Tr(\theta_{c}(E^*))\in \Q$. Thus, $\theta_{\id}(A)\simeq\theta_{c}(A)=:\theta(A)$.
\item For $\sigma \in \Gal(L/F)$, let $r$ denote the order of $\pi(\sigma)$. Then $r$ is $1$, $2$, $3$, $4$, or $6$, and
$$
\Tr(\theta(A)(\sigma))=2+\zeta_r+\overline \zeta_r\,.
$$
\end{enumerate}
\end{lemma}

\begin{proof}[Proof of $i)$ and $ii)$ of Proposition \ref{proposition: M}]
Let $r$ denote the order (4 or 6) of $\tau$. Let $\tilde\tau\in \Gal(L/F)$ be such that $\pi(\tilde\tau)=\tau$. Let $\xi$, $\omega$ be the eigenvalues of $\theta_\iota(E^*)(\tilde\tau)$. Then, $i)$ and $ii)$ of Lemma \ref{lemma: FS14}, give
\begin{equation}\label{equation: ig}
(\xi+\omega)(\overline\xi+\overline \omega)=2+\zeta_r+\overline\zeta_r\,.
\end{equation}
Without loss of generality, we may reorder $\xi$ and $\omega$ so that (\ref{equation: ig}) implies that $\xi=\zeta_r\omega$. The condition
$$
\Tr(\theta_\iota(E^*)(\tilde\tau))=\omega(1+\zeta_r)\in M
$$
forces $\omega$ to belong to the biquadratic extension $M(\zeta_r)$. Therefore, there are only the following possibilities for the value of the order $t$ of $\omega$: 1, 2, 3, 4, 6, 8, 12.
Suppose that $r=4$. It is straightforward to check that:
\begin{itemize} 
\item if $t=1$, $2$, $4$, then $\omega(1+\zeta_r)\in\Q(\sqrt{-1})\setminus\Q$;
\item if $t=8$, then $\omega(1+\zeta_r)\in\Q(\sqrt{-2})\setminus\Q$;
\item the values $t=3$, $6$, $12$ are incompatible with $\omega(1+\zeta_r)$ belonging to a quadratic imaginary field. 
\end{itemize}
Suppose that $r=6$. Then one readily checks that
\begin{itemize}
\item if $t=1$, $2$, $3$, $6$, then $\omega(1+\zeta_r)\in\Q(\sqrt{-3})\setminus\Q$;
\item the values $t=4$, $8$, $12$ are incompatible with $\omega(1+\zeta_r)$ belonging to a quadratic imaginary field. 
\end{itemize}
\end{proof}

\begin{remark}
One could wonder whether $\Gal(K/F)$ containing an element~$\tau$ of order $r=3$ imposes some condition on~$M$. However, the argument of the proof of Proposition \ref{proposition: M} can not be applied, since $\omega(1+\zeta_r)$ make take the rational value $\zeta_3(1+\zeta_3)=-1$. 
\end{remark}

\begin{lemma}\label{lemma: LF}
One has:
\begin{enumerate}[i)]
\item If $(\alt 4)$ holds and $M\not =\Q(\sqrt {-1}),\Q(\sqrt{-3})$, then $\Gal(L/F)$ is isomorphic to the binary tetrahedral group $\mathrm{B_T}$.
\item If $(\sym 4)$ holds, then $\Gal(L/F)$ is isomorphic to either $\mathrm{B_O}$, $\GL_2(\Z/3\Z)$,  or \footnote{The Gap identification numbers of $\mathrm{B_T}$, $\mathrm{B_O}$, $\GL_2(\Z/3\Z)$, and $G_8$ are $\langle 24, 3\rangle$, $\langle 48,28\rangle$, $\langle 48,29\rangle$, and $\langle 96,67\rangle$, respectively.} $G_8$.
\end{enumerate}
\end{lemma}

\begin{proof}
In the course of this proof, let us say that a finite group is \emph{$2$-embeddable} if it possesses a faithful representation of dimension~$2$ with coefficients in a quadratic imaginary field.
In case $i)$, we have that $L/K$ is at most quadratic by Lemma~\ref{lemma: fieldL}. The existence of the faithful representation $\theta_\iota(E^*)\colon \Gal(L/F)\rightarrow \GL_2(M)$ implies that $\Gal(L/F)$ is $2$-embeddable. Since $\alt 4$ is not $2$-embeddable, we have that $L\not = K$. There are two extensions of $\alt 4$ by $\cyc 2$: $\mathrm{B_T}$ and $\alt 4\times \cyc 2$. The lemma then follows from the fact that $\alt 4 \times \cyc 2$ is not $2$-embeddable.

As for $ii)$, we first note that $\mathrm{(H_4)}$ holds; then $M=\Q(\sqrt{-1})$ or $M=\Q(\sqrt{-2})$ by statement $i)$ of Proposition~\ref{proposition: M}. In particular Lemma~\ref{lemma: fieldL} implies that $\Gal(L/K)\subseteq \cyc 4$. There are eight groups that are extensions of~$\sym 4$ by~$\cyc 4$; among them, only $G_8$ is $2$-embeddable. Since $\sym 4$ is not $2$-embeddable, we are left with the case that $L/K$ is quadratic. There are four extensions of~$\sym 4$ by~$\cyc 2$, only two of which are $2$-embeddable: $\mathrm{B_O}$ and $\GL_2(\Z/3\Z)$. 
\end{proof}

\begin{proof}[Proof of $iii)$ and $iv)$ of Proposition \ref{proposition: M}]
Assume that we are in case $iii)$. We may suppose that $M\not=\Q(\sqrt{-1}),\Q(\sqrt{-3})$, since in this case the proposition is trivially true. Then, by Lemma~\ref{lemma: LF} we have a faithful representation
$$
\theta_\iota(E^*)\colon \mathrm{B_T}\rightarrow \GL_2(M)\,.
$$
Inspecting the character table of $\mathrm{B_T}$, we realize that it has three faithful representations of dimension~$2$: one has rational trace (call it $\varrho$), and the other two have trace taking values in $\Q(\sqrt{-3})\setminus \Q$ on elements of order $3$ and $6$. Since we have assumed that $M\not=\Q(\sqrt{-3})$, we have that $\theta_\iota(E^*)\simeq \varrho$. Let $Q$ denote the subgroup of $\mathrm{B_T}$ isomorphic to the quaternion group. It is well-known that the restriction of $\varrho$ to $Q$, sometimes called the quaternionic representation of~$Q$, although having rational trace, is realizable over a field $M$ if and only if $M$ splits the quaternion algebra $(-1,-1)_\Q$ (see \cite[Cor. to Prop. 35]{Ser98}).  

Case $iv)$ is immediate from Lemma~\ref{lemma: LF}: one readily checks that the trace of any faithful representation of dimension~$2$ of $\mathrm{B_O}$ or $\GL_2(\Z/3\Z)$ takes values in $\Q(\sqrt{-2})\setminus \Q$ on elements of order~$8$, and that the trace of a faithful representation of dimension 2 of $G_8$ with traces in a quadratic imaginary field takes, in fact, values in $\Q(\sqrt{-1})\setminus\Q$ on elements of order 12.
\end{proof}

\subsection{Restriction of scalars: cohomological obstructions}\label{sec: restriction of scalars}
In this section, let $M$ be an imaginary quadratic field of discriminant $D$. Assume that
\begin{enumerate}[(0)]
\item $D$ is different from $ -3,-4$, and $-8$.
\end{enumerate} 
Suppose that $A/k$ is an abelian surface satisfying the following conditions:
\begin{enumerate}[(1)]
\item $k$ contains $M$;
\item $A_K\sim E^2$, where $E/K$ is an elliptic curve with CM by $M$;
\item the field $K$ is the smallest such that $\End(A_K)=\End(A_\Qbar)$; and
\item the Galois group $G=\Gal(K/k)$ is isomorphic to $ \sym 4$.
\end{enumerate}
Note that $E$ is then a $k$-curve completely defined over $K$ (cf. Definition \ref{def: completely defined}). Let $\gamma_E^K\in H^2(G,M^*)$ be the cohomology class attached to $E$ as in \eqref{eq: completely defined}. Since $K$ will be fixed through this section, let us denote $\gamma_E^K$ simply by $\gamma_E$.

In \S\ref{section: galstr} we have found that the above assumptions impose conditions on the field $M$. For example, since condition $\mathrm{(A_4)}$ holds, Proposition~\ref{proposition: M} implies that $M$ necessarily splits the quaternion algebra $(-1,-1)_\Q$. The goal of this section is to give some more necessary conditions on $M$, under the following additional condition (which we assume from now on):
\begin{itemize}
\item[(NE)] If $D$ is even, then it is either divisible by $8$ or by some prime $p\equiv 3\mod 4$.
\end{itemize}
Recall the model $E^*$ of $E$ up to isogeny introduced in Remark \ref{rk: field of def}. It is any elliptic curve with CM by the maximal order of $M$ and defined over the field $F = k(j_{E^*})$. We next show that condition $\mathrm{(NE)}$ allows for a particular choice of $E^*$ that will be key to our computations.

First of all we note that $[F\colon k]=2$. Indeed, $F/k$ is a polyquadratic extension and, since $F\subseteq K$, we see that $[F\colon k]$ is either $1$ or $2$.  If $F=k$, then  Proposition~\ref{proposition: M} implies that $M=\Q(\sqrt{-2})$, which contradicts our assumption $(0)$; therefore we must have $[F\colon k]=2$. Let $H$ be the Hilbert class field of $M$. Thanks to the assumption $\mathrm{(NE)}$, by \cite[\S 11]{Gr80} there exists an elliptic curve $\tilde E/H$ with CM by $M$ that is $H$-isogenous to its $\Gal(H/\Q)$-conjugates. Since $H\subseteq F$, we can (and do) take as~$E^*$ the curve~$\tilde{E}_F$. 

For this particular choice, $E^*$ is a $k$-curve completely defined over $F$; namely,~$E^*$ is $F$-isogenous to its $\Gal(F/k)$-conjugate. In particular, the cohomology class $\gamma_{E^*}=\gamma_{E^*}^K$ lies in the image of the inflation map $$H^2(\Gal(F/k),M^*)\ra H^2(G,M^*),$$ and a cocycle $c_{E^*}$ representing $\gamma_{E^*}$ is of the form
  \begin{align}\label{eq:C_E_star }
    c_{E^*}(\sigma,\tau)= \begin{cases} m \text{ if } \sigma_{|F}\neq \mathrm{Id} \text{ and }\tau_{|F}\neq \mathrm{Id},\\ 1 \text{ otherwise,} \end{cases}
  \end{align}
for some $m\in M^*$. Observe that $\gamma_{E^*}$ only depends on the class of $m\mod{(M^*)^2}$, for replacing $c_{E^*}$ by a cohomologous cocycle changes $m$ by an element of $(M^*)^2$.

The aim of this section is to prove the following result.
\begin{proposition}\label{prop:m}
  Either $2m$ or $-2m$ is a square in $M$.
\end{proposition}
The proof follows from an explicit computation of the decomposition into simple varieties of the variety $R = \Res_{K/k}E$, the restriction of scalars of $E$. This is justified by the observation that $A$ is one of the factors of such decomposition.
\begin{lemma}\label{eq: lemma A is a simple factor of R}
The abelian variety $A$ is a simple factor of $R$. More precisely, 
$$R\sim A^2\times A',$$
for some abelian variety $A'$ that does not contain any factor isogenous to $A$.
\end{lemma}
\begin{proof}
  By the universal property of the restriction of scalars we have an isomorphism of $\Q$-vector spaces
  \begin{align}
    \Hom(A,R)=\Hom(A_K,E)\simeq M^2.
  \end{align}
We claim that $\End(A)\simeq M$ as $\Q$-vector spaces (and, in fact, as $\Q$-algebras). Indeed, the Galois endomorphism type (as described in \cite[\S4]{FKRS12}) of $A$ is \textbf F[$\sym 4$], and then by \cite[Table 8]{FKRS12}, one deduces that $\End(A)$ has dimension $2$ over $\Q$. The isomorphism then follows from Proposition~\ref{proposition: centre racional}.

Therefore, we see that two (and only two) copies of $A$ appear in the decomposition of $R$ into simple varieties and this finishes the proof.
\end{proof}
In order to further determine the decomposition of $R$, we will compute the decomposition of its endomorphism algebra into simple algebras. By \cite[\S 15]{Gr80} (cf. also \cite[Lemma 6.4]{Rib92} for the non-CM case), there is an isomorphism of algebras 
\begin{align}
  \End(R)\simeq M^{c_{E}}[G]\,,
\end{align}
where $M^{c_{E}}[G]$ denotes the \emph{twisted group algebra} of $G$ by a cocycle $c_E$ representing~$\gamma_E$.
This is the $M$-algebra with $M$-basis the symbols $\{u_\sigma\}_{\sigma\in G}$ and multiplication given by the rule
\begin{align}\label{eq: multiplication in twisted group algebra}
  u_\sigma \cdot u_\tau = c_E(\sigma,\tau)u_{\sigma\tau}.
\end{align}
In order to relate $c_{E^*}$ and $c_{E}$ we need two basic lemmas.
\begin{lemma}\label{lemma:isogeny over quadratic}
  Let $C/K_1$ and $C'/K_1$ be elliptic curves with CM by a field $M_1$ contained in $K_1$ and different from $\Q(\sqrt{-1})$ and $\Q(\sqrt{-3})$. Then there exists an element $\beta\in K_1$ such that $C'$ is $K_1$-isogenous to $C_\beta$, the $K_1(\sqrt{\beta})$-twist of $C$.
\end{lemma}
\begin{proof}
  Since $C$ and $C'$ have CM by $M_1$ there exists a $\Qbar$-isogeny $\lambda\colon C'_{\Qbar}\ra C_{\Qbar}$. The fact that $G_{K_1}$ acts trivially on $M_1$ implies that the map
  \begin{align}\label{eq: Homs G_K}
\begin{array}{ccc}
    G_{K_1} & \lra & M_1^*\\
\sigma & \longmapsto & \lambda^{-1}\circ \acc\sigma \lambda
\end{array}
  \end{align}
is a homomorphism. Its image is finite, because $\lambda$ is defined over some finite extension of $K_1$, and therefore contained in $\{\pm 1\}$ by our hypothesis on $M_1$. If the image is trivial then $\lambda$ is defined over $K_1$; otherwise, it is defined over a field of the form $L_1=K_1(\sqrt{\beta})$ for some $\beta\in K_1^*\setminus (K_1^*)^2$. The set $\Hom(\Gal(L_1/K_1),M_1^*)= H^1(\Gal(L_1/K_1),M_1^*)$ parametrizes the elliptic curves which are $L_1$-isogenous to $C$, up to $K_1$-isogeny. The curve $C_\beta$ also corresponds to the homomorphism \eqref{eq: Homs G_K}, and therefore $C'$ is $K_1$-isogenous to $C_\beta$.
\end{proof}

In the following statement we use the interpretation of $H^2(G,\{\pm 1\})$ as a group classifying (classes of) central extensions of $G$ by $\{\pm 1\}$.
\begin{lemma}\label{lemma: completely def}
  Let $C$ be a $k_1$-curve completely defined over a field $K_1$. Suppose that $C$ has CM by a field $M_1$  contained in $ k_1$, and let $\gamma_C\in H^2(\Gal(K_1/k_1), M_1^*)$ be the corresponding cohomology class. Let $L_1=K_1(\sqrt{\beta})$ for some $\beta\in K_1$, and let $C_\beta$ be the $L_1$-twist of $C$. The curve $C_\beta$ is completely defined over $K_1$ if and only if $L_1$ is Galois over $k_1$. Moreover, in that case the cohomology classes $\gamma_C,\gamma_{C_\beta}\in H^2(\Gal(K_1/k_1,M_1^*))$ satisfy the relation
  \begin{align*}
    \gamma_{C_\beta} = \gamma_C \cdot \gamma_{L_1},
  \end{align*}
where $\gamma_{L_1}\in H^2(\Gal(K_1/k_1),\{\pm 1\})$ stands for the cohomology class attached to 
\begin{align*}
  1 \lra \Gal(L_1/K_1)\simeq \{\pm 1\} \lra \Gal(L_1/k_1)\lra \Gal(K_1/k_1)\lra 1.
\end{align*}
\end{lemma}
\begin{proof}
  This is proved in \cite[Lemma 6.1]{GQ14} in the non-CM case. The same argument works in the CM case, with just a small remark: in the course of the argument of loc. cit. one constructs isogenies $\nu_\sigma\colon \acc\sigma C_\beta\ra C_\beta$ and uses the fact that if the curve is completely defined over $K_1$ then these $\nu_\sigma$ are necessarily defined over $K_1$. This is obvious in the non-CM case; in the CM case, if follows from our hypothesis that $M_1\subseteq k_1$, because then all the endomorphisms of $C_\beta$ are defined over $K_1$.
\end{proof}

Since $E$ and $E^*$ are $k$-curves completely defined over $K$, the above lemmas imply that $E$ and $E^*$ become isogenous over a quadratic extension $L/K$ such that $L/k$ is Galois, and that
\begin{align}\label{eq: gamma e and gamma l}
  \gamma_E = \gamma_{E^*}\cdot \gamma_L,
\end{align}
where $\gamma_L$ is the cohomology class attached $L/K$. 
\begin{remark}
Observe that $\gamma_L$ in principle belongs to $H^2(G,\{\pm 1\})$; in expressions like \eqref{eq: gamma e and gamma l}, we  make the slight abuse of notation of using the same symbol to denote the image of $\gamma_L$ under the natural map
\begin{align}\label{eq: natural map in H2}
  H^2(G,\{\pm 1\})\ra H^2(G,M^*)[2].
\end{align}
This is justified by the fact that \eqref{eq: natural map in H2} is injective. This follows, for instance, from the following decomposition, which is analogous to \eqref{equation: dec}:
\begin{align*}
  H^2(G,M^*)[2]\simeq H^2(G,\{\pm 1\})\times H^2(G,M^*/\{\pm 1\})[2].
\end{align*}
\end{remark}

Recall that $G\simeq \sym 4$. The cohomology group $H^2(\sym 4,\{\pm 1\})$ is known to be isomorphic to $ \cyc 2 \times \cyc 2$ (see, e.g., \cite[\S 1.5]{Se84} and the references therein). Besides the trivial cohomology class there is one symmetric cohomology class and two non-symmetric ones. The corresponding possibilities for $\Gal(L/k)$ are also well-known, and we are only interested in the non-symmetric classes, which correspond to $\mathrm{B_O}$ (the binary octahedral group introduced in \S\ref{section: galstr}) and to $\GL_2(\Z/3\Z)$.
\begin{lemma}\label{lemma: nonsym}
 The cohomology class $\gamma_L$ in \eqref{eq: gamma e and gamma l} is non-symmetric. That is, $\Gal(L/k)$ is isomorphic to $\mathrm{B_O}$ or $\GL_2(\Z/3\Z)$.
\end{lemma}
\begin{proof}
Suppose that $\gamma_L$ is symmetric. It turns out that the symmetric classes in $H^2(\sym 4,\{\pm 1\})$ are precisely those that have trivial restriction to $\alt 4$. Since $\Gal(K/F)\simeq \alt 4$, and $\gamma_{E^*}$ has trivial restriction to $\Gal(K/F)$, we see that $\gamma_E$ has trivial restriction to $\Gal(K/F)$. Therefore, $E$ is $K$-isogenous to a curve defined over~$F$ (cf. Remark \ref{rk: K-isogenous to a curve over F}). But this is a contradiction with Lemma~\ref{lemma: LF}, which implies that the minimal extension of~$K$ over which this can happen is non-trivial.
\end{proof}

From now on we denote by $\gamma_-\in H^2(G,\{\pm 1\})$ the class corresponding to $\mathrm{B_O}$ and by $\gamma_+$ the one corresponding to $\GL_2(\Z/3\Z)$ (and, as usual $c_+$ and $c_-$ denote cocycles representing them). So far we have seen that 
\begin{align}
  \label{eq:gammaeandgammal}
  \gamma_E = \gamma_{E^*}\cdot \gamma_\pm,\ \ \text{where $\gamma_{E^*}$ is given by \eqref{eq:C_E_star } and $\gamma_\pm\in \{\gamma_+,\gamma_-\}$.}
\end{align}
The next step is to compute the center of $M^{c_E}[G]$. As a previous calculation, we need to determine the center of $M^{c_{\pm}}[G]$.  For this we will use the characterization of the center of a twisted group algebra $M^c[G]$ given in \cite[p. 321]{Ka87}. Given a two-cocycle $c\in Z^2(G,M^*)$, an element $g\in G$ is said to be \emph{$c$-regular} if
\begin{align}\label{eq: regular }
  c(g,h)=c(h,g)\ \ \text{ for all } h \text{ in the centralizer } C_G(g) \text{ of $g$}.
\end{align}
All conjugates of a $c$-regular element are also $c$-regular. Let $X$ denote a set of representatives of the $c$-regular conjugacy classes. For $x\in X$ let $T_x$ denote a system of representatives of $G/C_G(x)$. For every $x\in X$, the element 
\begin{align}\label{eq: formula kx initial}
  k_x = \sum_{g\in T_x} u_g u_x u_g^{-1}
\end{align}
 belongs to the center of $M^{c}[G]$. Moreover,  $\{k_x\}_{x\in X}$ is an $M$-basis of the center. Observe that, by making use of \eqref{eq: multiplication in twisted group algebra}, the element $k_x$ can also be expressed as
\begin{align}\label{eq: formula kx }
   k_x = \sum_{g\in T_x} u_g u_x u_g^{-1} =  \sum_{g\in T_x}c(gx,g^{-1})c(g,x)c(g,g^{-1})^{-1} u_{gxg^{-1}}.
\end{align}

\begin{lemma}
  The center of $M^{c_{\pm}}[G]$ is isomorphic to $M\times M[t]/(t^2 \mp 2)$.
\end{lemma}
\begin{proof}
This is an explicit computation with the cocycle $c_\pm$. The computation is elementary, but lengthy. For this reason, we have used the software Magma \cite{magma} to carry it out.\footnote{In \url{https://github.com/xguitart/sato-tate} the interested reader can find the Magma script that we used.} We reproduce here the details only for $c_+$; they are very similar for $c_-$. 

To begin with, Magma implements routines that allow for the explicit computation the cocycle $c_+$. Alternatively, one can compute it using the exact sequence 
\begin{align*}
  1\ra \{\pm 1\}\ra \GL_2(\Z/3\Z)\ra \sym 4\simeq G \ra 1
\end{align*}
as follows: For each $\sigma\in G$ fix a lift $\tilde \sigma \in \GL_2(\Z/3\Z)$; then $c_+(\sigma,\tau)=\tilde\sigma \cdot \tilde\tau\cdot \widetilde{\sigma\tau}^{-1}$.

Knowing the values $c_+(\sigma,\tau)$ for all $\sigma,\tau\in \sym 4$ (here we are identifying $G$ with $\sym 4$), one can compute the set $X$: it consists of the conjugacy classes of $\mathrm{Id}$, $(1,2,4)$, and $(1,4,3,2)$ . This already implies that the center has dimension $3$ over $M$.

It remains to compute the structure of the center as an algebra. For this, consider the element $k_y$ with $y=(1,4,3,2)$, which can be computed explicitly by means of \eqref{eq: formula kx } with $c=c_+$. It turns out that
\begin{small} 
 \begin{align}
    k_{y} =& \sum_{g\in T_y}c_+(gy,g^{-1})c_+(g,y)c_+(g,g^{-1})^{-1}u_{gyg^{-1}}\\ =& u_{(1,2,3,4)}+ u_{(1,4,3,2)}+ u_{(1,3,4,2)} +u_{(1,3,2,4)}-u_{(1,4,2,3)}-u_{(1,2,4,3)}.\label{eq: kx explicit}
  \end{align}
\end{small}
Using the multiplication formula in the twisted group algebra, namely $u_\sigma\cdot u_\tau = c_+(\sigma,\tau)u_{\sigma\tau}$, one computes that
\begin{small} 
 \begin{align*}
    k_{y}^2 & = 3(-u_{(1,3,2)}+ u_{(1,2,3)}- u_{(1,4,3)} +u_{(1,3,4)}-u_{(1,4,2)}+2u_\mathrm{Id}\\ &-u_{(2,4,2)} + u_{(1,2,3)}-u_{(2,3,4)})
  \end{align*}
\end{small}
and
\begin{small} 
 \begin{align*}
    k_{y}^3 = 18(u_{(1,2,3,4)}+ u_{(1,4,3,2)}+ u_{(1,3,4,2)}+u_{(1,3,2,4)}-u_{(1,4,2,3)}-u_{(1,2,4,3)}).
  \end{align*}
\end{small}
We see that  $k_y^3-18k_y=0$ and that the powers of $k_y$ do not satisfy any linear relation of lower degree. Thus the minimal polynomial of $k_y$ is $t(t^2-18)$ which implies that the center of $M^{c_\pm}[G]$ is isomorphic to $M\times M[t]/(t^2 - 2)$. 

The same calculation for $c_-$ gives the minimal polynomial $t(t^2+18)$.
\end{proof}
\begin{lemma}\label{lemma:the center}
  The center of $M^{c_E}[G]$ is isomorphic to $M\times M[t]/(t^2\mp 2m)$.
\end{lemma}
\begin{proof}
 Again we just explicit the calculations for $c_+$; the case of $c_-$ is analogous. By \eqref{eq:gammaeandgammal} we can assume that $c_E = c_{E^*}\cdot c_+$. Recall that $c_{E^*}$ is a symmetric cocycle (cf. formula \eqref{eq:C_E_star }) which lies in the image of the inflation map $$H^2(\Gal(F/k),M^*)\ra H^2(G,M^*).$$ Since $c_{E^*}$ is symmetric, in view of \eqref{eq: regular } an element $g\in G$ is $c_E$-regular if and only if it is $c_+$-regular. This implies that the center of $M^{c_E}[G]$ has the same dimension as the center of $M^{c_+}[G]$. We next determine its algebra structure.

One of the elements of the center of $M^{c_E}[G]$ is
  \begin{small} 
 \begin{align*}
    \tilde k_{y} =& \sum_{g\in T_y}c_E(gy,g^{-1})c_E(g,y)c_E(g,g^{-1})^{-1}u_{gyg^{-1}},
  \end{align*}
\end{small}
where we can take $y= (1,4,2,3)$ as before. Using formula \eqref{eq:C_E_star } and the fact that $g\not \in \alt 4$ for $g\in T_y$ it is easy to check that $$c_{E^*}(gy,g^{-1})c_{E^*}(g,y)c_{E^*}(g,g^{-1})^{-1}=1.$$ Therefore, we have that
\begin{small} 
 \begin{align}
    \tilde k_{y} =& \sum_{g\in T_y}c_+(gy,g^{-1})c_+(g,y)c_+(g,g^{-1})^{-1}u_{gyg^{-1}}\\ =& \label{eq: ky tilde}
u_{(1,2,3,4)}+ u_{(1,4,3,2)}+ u_{(1,3,4,2)} +u_{(1,3,2,4)}-u_{(1,4,2,3)}-u_{(1,2,4,3)}.
  \end{align}
\end{small}
Observe that $\tilde k_y$ has the same expression as the element $k_y$ found in \eqref{eq: kx explicit} in terms of the basis $\{u_\sigma\}_{\sigma \in G}$. We remark that the two elements lie in different algebras though: $k_y$ lies in $M^{c_\pm}[G]$ and $\tilde k_y$ lies in $M^{c_E}[G]$. Thus, in order to compute $\tilde k_y^2$ we need to use now the multiplication of the twisted group algebra $M^{c_E}[G]$: $$u_\sigma\cdot u_\tau = c_E(\sigma,\tau)u_{\sigma\tau}.$$ But observe that the group elements appearing in \eqref{eq: ky tilde} do not belong to $\alt 4$. Therefore, in order to compute $\tilde k_y^2$ only the values $c_E(\sigma,\tau)$ with $\sigma,\tau\not \in \alt 4$ are involved in the calculation. For $\sigma,\tau\not\in A_4$ we have that $c_E(\sigma,\tau)=m\cdot c_+(\sigma,\tau)$. This means that $\tilde k_y^2$ has the same expression in terms of the basis as $m\cdot k_y^2$; that is
\begin{small} 
 \begin{align*}
    \tilde k_{y}^2 & = 3 m (-u_{(1,3,2)}+ u_{(1,2,3)}- u_{(1,4,3)} +u_{(1,3,4)}-u_{(1,4,2)}+2u_\mathrm{Id}\\ &-u_{(2,4,2)} + u_{(1,2,3)}-u_{(2,3,4)})
  \end{align*}
\end{small}

Now the group elements  appearing in the above expression of $\tilde k_y^2$ belong to $A_4$. Therefore, in the product $\tilde k_y^3 = \tilde k_y^2\cdot \tilde k_y$ all the products of basis elements are of the form $u_\sigma\cdot u_\tau$ with $\sigma\in \alt 4$ and $\tau \not \in \alt 4$, so that $c_E(\sigma,\tau)=c_+ (\sigma,\tau)$. This implies that $\tilde k_y^3$ has the same expression in terms of the basis elements as $m\cdot k_y^3$, namely
\begin{small} 
 \begin{align*}
    \tilde k_{y}^3 = 18 m  (u_{(1,2,3,4)}+ u_{(1,4,3,2)}+ u_{(1,3,4,2)}+u_{(1,3,2,4)}-u_{(1,4,2,3)}-u_{(1,2,4,3)}).
  \end{align*}
\end{small}
 Therefore, $\tilde k_y^3 - 18m \tilde k_y= 0$ and the minimal polynomial of $\tilde k_y$ is $t^3- 18 mt$ as we aimed to see.
\end{proof}
From the above lemma we see that $M^{c_E}[G]$ decomposes into the product of either two or three simple algebras, depending on whether $\mp 2 m$ is a square in $M$ or not. Before determining the structure of the center, we record the following piece of information about $M^{c_E}[G]$, which we will use in \S\ref{section: an example} below. 

\begin{lemma}\label{lemma:contains M2}
  Each simple factor of $M^{c_E}[G]$ contains $\mathrm{M}_2(M)$ as a subalgebra.
\end{lemma}
\begin{proof}
  The first step is to prove that $M^{c_E}[G]$ contains a subalgebra isomorphic to $\mathrm{M}_2(M)$. Indeed, let $H$ be the unique normal subgroup of $G$ isomorphic to $\cyc 2\times \cyc 2$. The only normal subgroup of $\sym 4$ isomorphic to $\cyc 2 \times \cyc 2$ is contained in $\alt 4$; since $\Gal(K/F)\simeq \alt 4$ we see that $H\subseteq \Gal(K/F)$. Thus $c_E(x,y)=c_\pm(x,y)$ for $x,y\in H$ because $c_{E^*}$ has trivial restriction to $\Gal(K/F)$. 

Now let $s$ and $t$ be generators of $H$. The elements $u_1$, $u_s$, $u_t$, and $u_{st}$ generate a subalgebra of $M^{c_E}[G]$ isomorphic to $(a,b)_M$ with $a,b\in\{\pm 1\}$. Indeed, one can check that
\begin{align}\label{eq: css and ctt}
  c_\pm(s,s) \in\{\pm 1 \}; \ c_\pm(t,t) \in \{ \pm 1\}; \ c_\pm(s,t)c_\pm(t,s)=-1,
\end{align}
so that $$u_s^2= \pm 1, \ u_t^2=\pm 1;\  u_su_t=-u_tu_s.$$ Therefore, $u_s$ and $u_t$ generate an algebra which is isomorphic to $(\pm 1,\pm 1 )_M$. As we have observed before, by Proposition \ref{proposition: M} $iii)$ the field $M$ splits the algebra $(-1,-1)_\Q$, so that $(\pm 1,\pm 1)_M\simeq\M_2(M)$.

Let $B$ be a simple factor of $M^{c_E}[G]$ with projection $\pi\colon M^{c_E}[G]\ra B$. The composition $$\M_2(M)\hookrightarrow M^{c_E}[G]\stackrel{\pi}{\ra} B$$ is a homomorphism of $M$-algebras. Since $\M_2(M)$ is simple, it does not have non-trivial ideals and the kernel of the above composition is trivial. Thus the homomorphism $\M_2(M)\ra B$ is injective and $\M_2(M)$ is a subalgebra of the simple factor~$B$.
\end{proof}
\begin{remark}\label{rk: validity of these lemmas in general}
  Observe that in the proof of Lemmas~\ref{lemma:the center} and~\ref{lemma:contains M2} we have not used that $E$ is the elliptic quotient of $A$. This will play a role in Proposition~\ref{prop: center splits} below, but the statements of Lemma \ref{lemma:the center} and Lemma \ref{lemma:contains M2} are valid for any $k$-curve completely defined over $K$, with $\Gal(K/k)\simeq \sym 4$ and whose cohomology class is as in \eqref{eq:gammaeandgammal}. This will be used in \S \ref{section: an example}.
\end{remark}
  \begin{proposition}\label{prop: center splits}
 The center of $M^{c_E}[G]$ is isomorphic to $M\times M\times M$.
 \end{proposition}
 \begin{proof}
 From Lemma \ref{lemma:the center} the center of $M^{c_E}[G]$ is isomorphic to either $M\times M \times M$ or $M\times L$, with $[L\colon M]=2$. Suppose that the center is $M\times L$. This implies that $$M^{c_E}[G]\simeq M_{r_1}(B_{1})\times M_{r_2}(B_{2}),$$ where $B_{1}$ stands for an $M$-central division algebra, say of Schur index $c_1$, and $B_{2}$ for an $L$-central division algebra of Schur index $c_2$. Since $M^{c_E}[G]$ has dimension~$24$ over $M$ we have that
 \begin{align}\label{eq: equals 24}
   r_1^2 c_1^2 + 2 r_2^2 c_2^2 = 24.
 \end{align}
 Now by Lemma \ref{eq: lemma A is a simple factor of R} one of the simple factors of $M^{c_E}[G]$ is isomorphic to $\M_2(M)$. This forces $r_1=2$ and $c_1=1$, but then \eqref{eq: equals 24} does not have any solution which is a contradiction.
\end{proof}

\begin{proof}
[Proof of Proposition \ref{prop:m}]
The statement is now a direct consequence of Proposition \ref{prop: center splits} and Lemma \ref{lemma:the center}. 
\end{proof}

{\bf Some explicit computations.} To illustrate the usefulness of Proposition \ref{prop:m} we have computed the cohomology class associated to certain curves with CM by a field $M$ of class number $2$, in the particular case where $k=M$. The calculations that we next describe have been done using Sage.\footnote{The reader can find the scripts in \url{https://github.com/xguitart/sato-tate}.} 

Let $M$ be one of the imaginary quadratic fields of Table \ref{table: ms}. Let $j_0$ and $j_1$ be the roots of the Hilbert class polynomial attached to the discriminant~$D$ of~$M$. The field $\Q(j_0)$ is real quadratic and the Hilbert class field of $M$ is given by $H=M\cdot \Q(j_0)$. In particular, since we take $k=M$, the field~$F$ coincides with~$H$.

As a first step,  we have computed an elliptic curve $E_0/\Q(j_0)$ with $j$-invariant~$j_0$ (this is easy; one can use for instance the explicit formula of \cite[Prop. 1.4]{Sil86}). Then, with a simple search method we have been able to find an element $\beta\in \Q(j_0)$ such that the curve $(E_0)_\beta$, the twist of $E_0$ by $\beta$, is $F$-isogenous to its $\Gal(F/k)$-conjugate.\footnote{Since $D$ satisfies $\mathrm{(NE)}$ there exists a curve over $F$, with CM by $M$, and which is $F$-isogenous to its $\Gal(F/k)$-conjugate. However, it is not clear to us that one can always find a model over $\Q(j_0)$ of such curve, as we did for the curves of Table \ref{table: ms}.} Put $E^* = (E_0)_\beta$, and denote by $\sigma$ a generator of $\Gal(F/k)$. 

Using Sage routines we have computed an isogeny $\mu_\sigma \colon  E^* \ra \acc \sigma E^*$ explicitly; that is, we have found the rational functions that define $\mu_\sigma$. Then it is easy to compute its Galois conjugate $\acc\sigma\mu_\sigma$ and the composition $\acc\sigma\mu_\sigma\circ\mu_\sigma$. Having explicitly the rational functions giving the isogeny $\acc\sigma\mu_\sigma\circ\mu_\sigma$ allows for the calculation of the kernel polynomial of this isogeny. In each of the entries of Table \ref{table: ms} we have checked that such kernel polynomial agrees with the kernel polynomial of the isogeny ``multiplication by $m$'', where $m$ is the one displayed in the second column. This means that $\acc\sigma\mu_\sigma\circ\mu_\sigma$ equals multiplication by $\pm m$ (the indeterminacy in the sign comes from the fact that the kernel polynomial determines the isogeny up to composing with $-1$). We remark that the isogeny $\mu_\sigma$ that we begin with is non-canonical; a different choice of $\mu_\sigma$ would change $m$ by an element of $(M^*)^2$.

As an example, we give the computations for the case $k=M=\Q(\sqrt{-40})$. The Hilbert class field in this case is $F=\Q(\sqrt{-40},\sqrt{5})$. The curve 
\begin{align*}
E^*\colon  y^2 = &\ x^3 + (135\sqrt{5} - 1125)x + 6480\sqrt{5} - 54000
\end{align*}
has CM by the ring of integers of $M$. There is an isogeny $\mu_\sigma \colon E^*_F\ra \acc\sigma E^*_F$ with kernel polynomial  $$x + 6\sqrt{5} - 30.$$  The Galois conjugate $\acc\sigma\mu_\sigma$ has kernel polynomial $$x - 6\sqrt{5} - 30.$$ One can check that the kernel polynomial of $\acc\sigma\mu_\sigma\circ \mu_\sigma$ equals the kernel polynomial of the multiplication by $2$ map on $E^*$.

On Table~\ref{table: ms}, we have computed the value of $m\mod{(M^*)^2}$ for some discriminants~$D$ of quadratic imaginary fields $M$ of class number $2$. Note that for $M=\Q(\sqrt{-35})$, $\Q(\sqrt{-51})$, or $\Q(\sqrt{-115})$, none of $2m$ or $-2m$ is a square of $M^*$, and thus Proposition \ref{prop:m} provides an obstruction to the existence of an abelian surface $A/k$ with $\Gal(K/k)\simeq \sym 4$ and the elliptic quotient $E$ having CM by~$M$. Observe, however, that for $D=-40$ and $D=-24$, we have that either $2m$ or $-2m$ is a square in $M$, since $m=\pm 2$. Therefore Proposition~\ref{prop:m} does not yield any obstruction in these cases. In fact, in \S\ref{section: an example} below we will exhibit an example of an abelian surface $A/k$ with $\Gal(K/k)\simeq \sym 4$ and the elliptic quotient $E$ having CM by~$\Q(\sqrt{-40})$.
\begin{table}[h]
\begin{center}
\begin{small}
\begin{tabular}{c|c}\hline
$D$ (discriminant of $M$) & $ m\mod{(M^*)^2} $ \\ \hline
$-24$ & $\pm 2$\\
$-35$ & $\pm 5$\\
$-40$ & $\pm 2$\\
$-51$ & $\pm 3$\\
$-115$ & $ \pm 5$\\\hline
\end{tabular}
\vspace{6pt}
\caption{Values of $m\mod{(M^*)^2} $ for certain curves $E^*$ with CM by fields of class number $2$.}\label{table: ms}
\end{small}
\end{center}
\end{table}

\subsection{Abelian surfaces over $\Q$}\label{section: imquad}

In this section, let $A$ be an abelian surface defined over $k$ satisfying that $A_{\Qbar}\sim E^2$, where $E$ is an elliptic curve over $\Qbar$ with CM, say by a quadratic imaginary field~$M$. Let $K/k$ be the minimal extension such that $\End(A_K)\simeq \End(A_\Qbar)$. 
Denote by $\Mm^{1}$ (resp. $\Mm^{2}$) the finite set of quadratic imaginary fields of class number $1$ (resp. of class number $2$). We also denote by $\Mm^{2,2}$ the finite set of imaginary quadratic fields with class group isomorphic to $\cyc 2 \times \cyc 2$. On Table \ref{table: Ms} below we list the discriminants $D$ of the quadratic imaginary fields in $\Mm^1$, $\Mm^2$, and $\Mm^{2,2}$, respectively. 

\begin{table}[h]
\begin{center}
\begin{small}
\begin{tabular}{c|c}\hline
$\Mm^1$  &$-3,-4,-7,-8,-11,-19,-43,-67,-163$\\ \hline
$\Mm^2$ & $ -15,-20,-24,-35,-40,-51,-52,-88,-91,-115$\\
$ $ & $-123,-148,-187,-232,-235,-267,-403,-427$\\ \hline
$\Mm^{2,2}$ & $-84,-120,-132,-168,-195,-228,-280,-312,-340,-372,-408,-435$\\
$  $ &$-483,-520,-532,-555,-595,-627,-708,-715,-760,-795,-1012,-1435$\\\hline
\end{tabular}
\vspace{6pt}
\caption{Discriminants of the imaginary quadratic fields with class group isomorphic to $\cyc 1$, $\cyc 2$, and $\cyc 2 \times \cyc 2$.}\label{table: Ms}
\end{small}
\end{center}
\end{table}
\begin{theorem}\label{theorem: absurfQ} Let $A$ be an abelian surface defined over a number field $k$ that is $\Qbar$-isogenous to the square $E^2$ of an elliptic curve defined over $\Qbar$ with CM by $M$. If $k$ is either $\Q$ or $M$, the set of possibilities for $M$ provided that $\Gal(K/M)\simeq G$ is contained in $\Mm(G)$, where the set $\Mm(G)$ is as defined on Table \ref{table: M}.
\begin{table}[h]
\begin{center}
\setlength{\extrarowheight}{0.5pt}
\vspace{6pt}
\begin{tabular}{c|cc}\hline
$\Gal(K/M)$ & $\Mm(\Gal(K/M))$\\\hline
$\cyc 1$ & $\Mm^1$\\
$\cyc 2$ & $\Mm^1\cup \Mm^2$ \\
$\cyc 3$ & $\Mm^1$\\
$\cyc 4$ & $\{\Q(\sqrt{-1}),\Q(\sqrt{-2})\}\cup \Mm^2$\\
$\cyc 6$ & $\{\Q(\sqrt{-3})\}\cup \Mm^2$\\
$\dih 2$ & $\Mm^1\cup \Mm^2\cup \Mm^{2,2}$\\
$\dih 3$ & $\Mm^1\cup \Mm^2$\\
$\dih 4$ & $\{\Q(\sqrt{-1}),\Q(\sqrt{-2})\}\cup \Mm^2\cup \Mm^{2,2}$\\
$\dih 6$ & $\{\Q(\sqrt{-3})\}\cup \Mm^2\cup \Mm^{2,2}$\\
$\alt 4$ & $\Mm^1\setminus \{\Q(\sqrt{-7})\}$\\
$\sym 4$ & $\{\Q(\sqrt{-1}),\Q(\sqrt{-2})\}\cup \Mm^2\setminus \{\Q(\sqrt{-15}),\Q(\sqrt{-35}),\Q(\sqrt{-51}),\Q(\sqrt{-115})\}$\\\hline
\end{tabular}
\vspace{6pt}
\caption{Possibilities for the field $M$ depending on $\Gal(K/M)$.}\label{table: M}
\end{center}
\end{table}
\end{theorem}
\begin{proof}
  Suppose first that $k=M$. By Proposition \ref{proposition: biquad}, $E$ admits a model $E^*$ up to isogeny defined over a subextension $F/k$ of $K/k$, with $\Gal(F/k)\simeq \cyc 2^{r}$ for some $r\leq r(A)$, where $r(A)$ is as defined in Remark~\ref{remark: r(A)}. By Remark~\ref{rk: field of def} we can actually suppose that $E^*$ is  an elliptic curve with CM by the maximal order of $M$, and that $F=k(j(E^*))$. The fact that $k=M$ implies that $F$ is  the Hilbert class field of $M$.  Therefore, $\Gal(F/M)=\Gal(F/k)\simeq \cyc{2}^{r}$ with $r\leq r(A)$ and necessarily $M$ lies in $\Mm^1\cup \Mm^2\cup \Mm^{2,2}$. This proves the rows of Table \ref{table: M} corresponding to $\cyc 1$, $\cyc 2$, $\cyc 3$, $\dih 2$, and $\dih 3$.

The remaining rows except of the last one follow by combining the above result with Proposition~\ref{proposition: M}. Indeed, for rows $\cyc 4$ and $\dih 4$ observe that, if $M$ has class number $1$, then $F=M$ and $\Gal(K/F)$ contains an element of order $4$, so that $i)$ of Proposition~\ref{proposition: M} forces $M$ to be $\Q(\sqrt{-1})$ or $\Q(\sqrt{-2})$. Similarly, rows $\cyc 6$ and $\dih 6$ are a consequence of $ii)$ of Proposition \ref{proposition: M}. The row $\alt 4$ follows from~$iii)$: $\Q(\sqrt{-7})$ does not split $(-1,-1)_\Q$ so it can not appear in the list.

Finally, for the row $\sym 4$, besides from the above considerations, one needs to invoke Proposition~\ref{prop:m} and Table~\ref{table: ms}. Indeed, first note that if the class number is $1$, then $M$ can only be $\Q(\sqrt{-2})$ because of statement $iv)$ of Proposition~\ref{proposition: M}; second, $\Q(\sqrt{-15})$ can not appear, because it does not split $(-1,-1)_\Q$; third, the quadratic imaginary fields of class number $2$ and discriminants $D=-35$, $D = -51$, and $D= -115$ cannot occur because the respective values of $\pm 2m$ (as listed on Table~\ref{table: ms}) are not squares in $M$. 

This finishes the case where $k=M$. The case $k=\Q$ now follows by applying the previous case to the base change $A_M$.
\end{proof}
We note that Table \ref{table: M} is to be read as follows: if $M$ does not belong to $\mathcal{M}(G)$, then there does not exists any $A/k$ such that $\Gal(K/M)\simeq G$ and its absolutely simple factor has CM by $M$. We do not claim, however, that for any pair $(G,M)$ with $M\in \mathcal{M}(G)$ there does exist such an $A$. In other words, it might be that some of the pairs of $(G,M)$ in Table \ref{table: M} cannot be realized by any abelian surface. 

In \S \ref{section: allST}, we will illustrate several approaches that one can take to finding examples of abelian surfaces realizing a concrete pair $(G,M)$. For example, one can perform a search over equations of genus~$2$ curves and try to identify the endomorphism algebra of its Jacobian and the minimum field of definition of its endomorphisms; or one can look for equations over families of genus~$2$ curves parametrizing those with a specified group of automorphisms. A slightly less explicit method (in the sense that one does not get the equation of a genus $2$ curve out of it) is to find $A$ as a simple factor of the restriction of scalars of a suitable elliptic curve with CM by $M$. We illustrate this in the following remark, which sumarizes a result of T. Nakamura in \cite{Na04}.

\begin{remark}\label{remark: exampleclassgroupc2xc2}  The pair $(G,M)=(\cyc 2 \times \cyc 2, \Q(\sqrt{-84}))$ is realized by an abelian surface defined over $M$. Indeed, let $K$ be the Hilbert class field of $M$, so that $\Gal(K/M)\simeq \cyc 2 \times \cyc 2$, and let $E$ be an elliptic curve with CM by $M$ and with the property that $E$ is defined over $K$ and it is $K$-isogenous to all of its $\Gal(K/\Q)$-conjugates. Such elliptic curves exist; in fact, in \cite[p. 190]{Na04} it is shown that there are $8$ of them, up to $K$-isogeny. Let $R= \mathrm{Res}_{K/M}E$ denote the restriction of scalars, which is an abelian variety over $M$ of dimension $4$. Nakamura shows that $\End(R)$ is isomorphic to the quaternion algebra $ (a,b)_M$, where the pair $(a,b)$ is one of the entries in the table of \cite[p. 190]{Na04}. It turns out that, for all the entries in that table, there is an isomorphism $(a,b)_M\simeq\M_2(M)$, and therefore $R\sim A^2$ for some $A/M$ such that $A_K\sim E^2$. It follows from the theory of complex multiplication that $K$ is the smallest field of definition of the endomorphisms of $A$.
\end{remark}

A similar approach can be used to realize the pair $(\sym 4, \Q(\sqrt{-40}))$. The argument is a bit lengthier though, and we devote \S \ref{section: an example} below to it.
\subsection{An abelian surface with $\Gal(K/k)\simeq \sym 4$ and $M\not =\Q(\sqrt{-2})$}\label{section: an example}
In this section we provide an example of abelian surface $A/k$ that satisfies conditions $(0)$, $(1)$, $(2)$, $(3)$, and $(4)$ described at the beginning of \S \ref{sec: restriction of scalars}. The strategy is to begin with a suitable choice of the fields $K$, $k$, and $M$, and of the curve $E/K$. Then we construct~$A$ as a simple factor of the restriction of scalars of $E$, a step in which we will take advantage of the explicit calculations performed in \S \ref{sec: restriction of scalars}.

We begin by describing the number fields involved in the construction. Let $K_0$ be the number field with defining polynomial\footnote{This is the field \cite[\href{http://www.lmfdb.org/NumberField/4.0.5780.1}{Global Number Field 4.0.5780.1}]{LMFDB}, and it is contained in the degree $8$ field \cite[\href{http://www.lmfdb.org/NumberField/8.0.835210000.1}{Global Number Field 8.0.835210000.1}]{LMFDB}} $x^4-x^3+5x^2-5x+2$. The Galois closure $K_0'$ of $K_0$ has Galois group isomorphic to $\sym 4$ and contains $\Q(\sqrt{5})$ as a subfield. 

Consider also the imaginary quadratic field $M = \Q(\sqrt{-40})$. For the construction we take $k=M$ and $K=K_0'\cdot k$. The splitting field of the Hilbert class polynomial attached to $D=-40$ is $\Q(\sqrt{5})$, so that $F=M\cdot \Q(\sqrt{5})$ is the Hilbert class field of~$M$.  Thus the diagram of fields is the following:
\begin{align*}
\xymatrix{
         K \ar@{-}[d]_{12}      \\
     F = M\cdot \Q(\sqrt{5})         \\
  k=M=\Q(\sqrt{-40}) \ar@{-}[u]^{2} & \Q(\sqrt{5})
\ar@{-}[ul]^2\\
    &\Q \ar@{-}[ul]^2\ar@{-}[u]^2  
}
\end{align*}
with $\Gal(K/k)\simeq \sym 4$ and $\Gal(K/F)\simeq \alt 4$.  For future reference we also put $K_1=K_0\cdot k$ (note that $K_1$ is of degree $4$ over $k$ and its Galois closure is $K$).

The discriminant of $M$ satisfies condition $\mathrm{(NE)}$ of \S \ref{sec: restriction of scalars}, so there exists an elliptic curve $E^*/F$ which is $F$ isogenous to its $\Gal(F/k)$-conjugate. Thus $E^*$ is a $k$-curve completely defined over $F$ and one can attach to it a cohomology class $\gamma_{E^*}^F$, which lies in $H^2(\Gal(F/k),M^*)$ and it is non-trivial since $E^*$ does not admit a model up to isogeny over $k$. In fact, it will be more convenient for us to regard $E^*$ as a $k$-curve completely defined over $K$. In particular, we will be concerned with the cohomology class $\gamma_{E^*}^K\in H^2(\Gal(K/k),M^*)$, which is nothing more than the image of $\gamma_{E^*}^F$ under the inflation map $$H^2(\Gal(F/k),M^*)\ra H^2(\Gal(K/k),M^*).$$  
Throughout this section, let us simply write $\gamma_{E^*}=\gamma_{E^*}^K$. Therefore, as in \eqref{eq:C_E_star }, the cocycle $c_{E^*}$ representing $\gamma_{E^*}$ is of the form
  \begin{align}\label{eq:C_E_star}
    c_{E^*}(\sigma,\tau)= \begin{cases} m \text{ if } \sigma_{|F}\neq \mathrm{Id} \text{ and }\tau_{|F}\neq \mathrm{Id},\\ 1 \text{ otherwise,} \end{cases}
  \end{align}
for some $m\in M^*\setminus (M^*)^2$.  The cohomology class $\gamma_{E^*}$ is determined by the element~$m$ of \eqref{eq:C_E_star} (or rather by its class modulo $(M^*)^2$). This is precisely what we computed in the third entry of Table \ref{table: ms}; we record the result for future reference.
\begin{lemma}\label{lemma: m}
  The element $m$ of \eqref{eq:C_E_star} equals either $2$ or $-2$, up to multiplication by an element of $(M^*)^2$.
\end{lemma}

The curve $E^*$ is the starting point of our the construction of $A$. The next step is to consider a certain twist of $E^*$, associated to the solution of a suitable Galois embedding problem of $K/k$. Recall that the group $H^2(\sym 4,\{\pm 1\})$, which is isomorphic to $ \cyc 2\times \cyc 2$, classifies central extensions of $\sym 4$ by $\{\pm 1\}$. We are interested in the two cohomology classes $\gamma_+$ and $\gamma_-$, introduced in  \S\ref{sec: restriction of scalars}, that correspond to the extensions $\GL_2(\Z/3\Z)$ and $\mathrm{B_O}$. 

Given a cohomology class $\gamma\in H^2(\Gal(K/k),\{\pm 1\})$ a quadratic extension $L/K$ is said to be a solution to the embedding problem defined by $\gamma$ if the extension
\begin{align*}
  1\lra \{\pm 1\}\simeq \Gal(L/K)\lra \Gal(L/k)\stackrel{\mathrm{res}}{\lra} \Gal(K/k)\lra 1
\end{align*}
corresponds to the class of $\gamma$. From now on we regard $\gamma_+$ and $\gamma_-$ as elements of $H^2(\Gal(K/k),\{\pm 1\})$ by means of an identification $\Gal(K/k)\simeq \sym 4$. In this way, $\gamma_+$ and $\gamma_-$ define two embedding problems of $\Gal(K/k)$.

\begin{lemma}
  There exist solutions to the embedding problems associated to $\gamma_+$ and $\gamma_-$.
\end{lemma}
\begin{proof}
  Let $Q_{K_1}$ be the quadratic form $x\mapsto \mathrm{Tr}_{K_1/k}(x^2)$ and let $w(Q_{K_1})$ denote its Hasse--Witt invariant. We denote by $d_{K_1}$ the discriminant of $K_1/k$. By \cite[Thm. 3.8]{Qu95} the embedding problems corresponding to $\mathrm{B_O}$  and to $\GL_2(\Z/3\Z)$ are solvable if and only if the following elements in the Brauer group of $k$ are trivial:
  \begin{align}
    w(Q_{K_1})\otimes (2,d_{K_1}) \text{ and } w(Q_{K_1})\otimes (-2,d_{K_1}).
  \end{align}
 Since $K_1 = K_0\cdot k$ the above classes can be regarded as 
  \begin{align}
    w(Q_{K_0})_k \otimes (2,d_{K_0})_k \text{ and } w(Q_{K_0})_k\otimes (-2,d_{K_0})_k,
  \end{align}
where the index $k$ stands for the image under the natural map $\mathrm{Br}(\Q)\ra \mathrm{Br}(k)$. One checks that $w(Q_{K_0})$ is ramified only at $2$ and $\infty$, and therefore $w(Q_{K_0})_k$ is trivial (since $k$ is imaginary and $2$ ramifies in $k$). Also, the algebras $(\pm 2, d_{K_0})$ are split by $k$, so that $(\pm 2,d_{K_0})_k$ are trivial.
\end{proof}

Let $\gamma_\pm\in H^2(\Gal(K/k),\{\pm 1\})$ be the cohomology class defined as follows:
  \begin{align}\label{eq:gamma_pm}
    \gamma_\pm= \begin{cases} \gamma_+ \text{ if } m = 2 \mod (M^*)^2 \\ \gamma_- \text{ if } m = -2\mod (M^*)^2. \end{cases}
  \end{align}
Let $L = K(\sqrt{\beta})$ be a solution to the embedding problem associated to $\gamma_\pm$ and define  $E= (E^*)_\beta$, the quadratic twist of $E^*$ by $\beta$. Since $L/k$ is Galois, Lemma~\ref{lemma: completely def} implies that $E$ is a $k$-curve completely defined over $K$ with cohomology class 
  \begin{align}\label{eq: gammaE second}
    \gamma_{E} = \gamma_{E^*} \cdot \gamma_\pm.
  \end{align}
Let $R=\Res_{K/k} E$. In view of Remark~\ref{rk: validity of these lemmas in general}, the statements of Lemma \ref{lemma:the center} and Lemma \ref{lemma:contains M2} are valid in the present context, and they allow for the computation of $\End(R)$.
\begin{lemma}\label{lemma: choice of gamma}
   The center of $\End(R)$ is isomorphic to $M\times M \times M$,
\end{lemma}
\begin{proof}
  By Lemma \ref{lemma:the center} the center of $M^{c_E}[\Gal(K/k)]$ is isomorphic to 
  \begin{align}\label{eq:+2-2}
    M \times M[t]/(t^2 \mp 2m).
  \end{align}
By Lemma \ref{lemma: m} we have that $m=2$ or $m=-2$, and thanks to our choice of $\gamma_\pm$ in \eqref{eq:gamma_pm} we see that \eqref{eq:+2-2} becomes
\begin{align*}
  M\times M[t]/(t^2-4) \simeq M\times M\times M.
\end{align*}
\end{proof}

\begin{proposition}\label{prop:twisted_group_algebra}
One has:
\begin{align}\label{eq:dec into simple algebras}
\End(R)\simeq \M_2(M)\times \M_2(M)\times B,  
\end{align}
where $B$ is an $M$-central simple algebra of $M$-dimension $16$.
\end{proposition}
\begin{proof}
By the above lemma  $M^{c_E}[\Gal(K/k)]$ is the product of three $M$-central simple algebras. Denote by $r_1^2$, $r_2^2$, and $r_3^2$ the $M$-dimension of each simple factor. Since the dimension of $M^{c_E}[\Gal(K/k)]$ is $24$ we have that
\begin{align*}
  r_1^2 + r_2^2 + r_3^2 = 24.
\end{align*}
The only solution (up to permutation) is $r_1 = 2$, $r_2 = 2$ and $r_3 = 4$. That is to say, two factors have dimension $4$ and the other factor has dimension $16$. Since each simple factor contains $\M_2(M)$ by Lemma \ref{lemma:contains M2}, we see that the algebras of dimension~$4$ must be isomorphic to $\M_2(M)$, and this proves the proposition.
\end{proof}

The decomposition of $\End(R)$ provided by Proposition \ref{prop:twisted_group_algebra} implies an isotypical decomposition of $R$ of the form 
\begin{align}\label{eq:isotypic dec}
  R\sim A_1^2 \times A_2^2 \times A_3^r,
\end{align}
where the $A_i$ are simple abelian varieties, $A_1$ and $A_2$ correspond to the factors of the form $\M_2(M)$, and $r$ is equal to either $1$, $2$ or $4$. Let $A$ be either $A_1$ or $A_2$. We next show that $A$ satisfies the following properties:
\begin{itemize}
\item $\End(A)\simeq M$; 
\item $A_K\sim E^2$; and 
\item $K$ is the smallest extension of $k$ satisfying that $\End(A_K)=\End(A_\Qbar)$.
\end{itemize}
The first statement above is clear in light of \eqref{eq:dec into simple algebras}. We prove the other two in the following lemmas; this will finish the proof that $A$ is an abelian surface satisfying conditions $(0)$--$(4)$ stated at the beginning of \S \ref{sec: restriction of scalars}, as we aimed to see. 
\begin{lemma}
  $A$ is an abelian surface; that is, $A_K\sim E^2$.
\end{lemma}
\begin{proof}
  By the universal property of the restriction of scalars functor we have an isomorphism of vector spaces
  \begin{align*}
    \Hom(A,R)\simeq \Hom(A_K,E);
  \end{align*}
from \eqref{eq:isotypic dec} and \eqref{eq:dec into simple algebras} we see that $\Hom(A,R)\simeq M^2$ and therefore $\Hom(A_K,E)\simeq M^2$, which implies that $A_K\sim E^2$.
\end{proof}
\begin{lemma}
  There is no field $N$ with $k\subsetneq N \subsetneq K$ such that $\End(A_N)=\End(A_\Qbar)$.
\end{lemma}
\begin{proof}
  Suppose that such $N$ exists. Without loss of generality we can assume that it is the smallest field satisfying this property. This implies, in particular, that $N/k$ is a Galois extension. Then $A_N\sim C^2$, for some elliptic curve $C$ over~$N$. Since $A_K\sim E^2$ this implies that $C_K\sim E$. Therefore, the cohomology class $\gamma_E$ is trivial when restricted to $\Gal(K/N)$. But $\Gal(K/N)$ is a normal subgroup of $\Gal(K/k)\simeq \sym 4$, and all normal subgroups of $\sym 4$ are contained in $\alt 4$; thus $\Gal(K/N)$ is a normal subgroup of $\Gal(K/F)\simeq \alt 4$. Since $\gamma_E = \gamma_{E^*}\cdot \gamma_\pm$ and the restriction of $\gamma_{E^*}$ to $\Gal(K/F)$ is trivial, this implies that $\gamma_\pm$ restricted $\Gal(K/N)$ is trivial. This is a contradiction, because it turns out that neither $\gamma_+$ nor $\gamma_-$ become trivial when restricted to any normal subgroup of $\alt 4$. Indeed, let $H=\langle s,t \rangle$ denote the unique normal subgroup of $\Gal(K/k)$ isomorphic to $\cyc 2 \times \cyc 2$. By looking at the diagram of subgroups of $\sym 4$ we see that $\Gal(K/N)$ must contain $H$. But we have already seen in \eqref{eq: css and ctt} that $\gamma_\pm$ restricted to $H$ is non-trivial.
\end{proof}

\section{Two applications to Sato--Tate groups}\label{section: ST}

In this section, we give two applications of the results obtained so far to the theory of Sato--Tate groups of abelian surfaces over number fields. The Sato--Tate group of an abelian surface $A$ defined over the number field $k$, which we will denote by $\ST(A)$, is a compact real Lie subgroup of $\USp(4)$ that conjecturally  governs the distribution of the Frobenius elements in the image of $\ell$-adic representation attached to $A$. We refer to \cite[\S2]{FKRS12} for the precise construction of the Sato--Tate group of an abelian variety defined over a number field. The reader is referred to the original source \cite[Chap. 8]{Ser12} for a construction of the Sato--Tate group in a more general context.  

The main result of \cite{FKRS12} establishes the existence of~$52$ possibilities for the Sato--Tate group of an abelian surface defined over a number field, all of which occur for some choice of the number field and the abelian surface (see \cite[Thm. 4.3]{FKRS12}). This result is complemented by establishing a one-to-one correspondence between Sato--Tate groups of abelian surfaces and their \emph{Galois endomorphism types}, an algebraic structure which encaptures both the Galois action on the ring of endomorphisms of the abelian surface and the structure of this ring as an $\R$-algebra (see \cite[Def. 1.3]{FKRS12} for a precise definition). Let $N_{\mathrm{ST},2}(k)$ denote the number of subgroups of $\USp(4)$ up to conjugacy that arise as Sato--Tate groups of abelian surfaces defined over the number field $k$.

We will make use of the notations for Sato--Tate groups and Galois endomorphism types introduced in \cite[\S4]{FKRS12} from now on. On Table~\ref{table:STgroups} we show the dictionary between Sato--Tate groups and Galois endomorphism types. The groups decorated with a $\star$ are those that arise over~$\Q$. Note that we thus have $N_{\mathrm{ST},2}(\Q)=34$. 

\begin{table}[h]
\begin{center}
\small
\begin{tabular}{ll|ll}\hline
 $G$  & Galois type & $G$  & Galois type \\\hline\vspace{-10pt}\\
 $C_1$  & $\bF[\cyc{1}]$ & $D_{4,1}{}^\star$   & $\bF[\dih{4},\dih{2}]$  \\
 $C_2$  & $\bF[\cyc{2}]$ & $D_{6,1}{}^\star$  & $\bF[\dih{6},\dih{3},\M_2(\R)]$ \\
 $C_3$  & $\bF[\cyc{3}]$ & $D_{3,2}{}^\star$  & $\bF[\dih{3},\cyc{3}]$  \\
 $C_4$  & $\bF[\cyc{4}]$ & $D_{4,2}{}^\star$  & $\bF[\dih{4},\cyc{4}]$ \\
 $C_6$  & $\bF[\cyc{6}]$ & $D_{6,2}{}^\star$  & $\bF[\dih{6},\cyc{6}]$  \\
 $D_2$  & $\bF[\dih{2}]$ & $O_1{}^\star$  & $\bF[\sym{4},\alt{4}]$  \\
 $D_3$  & $\bF[\dih{3}]$ & $E_1{}^\star$  & $\bE[\cyc{1}]$ \\
 $D_4$  & $\bF[\dih{4}]$ & $E_2{}^\star$        &  $\bE[\cyc{2}, \C]$  \\
 $D_6$  & $\bF[\dih{6}]$ & $E_3{}^\star$  & $\bE[\cyc{3}]$\\
 $T$    & $\bF[\alt{4}]$ & $E_4{}^\star$  & $\bE[\cyc{4}]$ \\
 $O$    & $\bF[\sym{4}]$ & $E_6{}^\star$  & $\bE[\cyc{6}]$ \\
 $J(C_1)$ & $\bF[\cyc{2},\cyc{1},\HH]$ & $J(E_1){}^\star$  & $\bE[\cyc{2},\R\times\R]$  \\
 $J(C_2){}^\star$ & $\bF[\dih{2},\cyc{2},\HH]$ & $J(E_2){}^\star$  & $\bE[\dih{2}]$ \\
 $J(C_3)$ & $\bF[\cyc{6},\cyc{3},\HH]$ & $J(E_3){}^\star$  & $\bE[\dih{3}]$  \\
 $J(C_4){}^\star$ & $\bF[\cyc{4}\times\cyc{2},\cyc{4}]$ & $J(E_4){}^\star$   & $\bE[\dih{4}]$ \\
 $J(C_6){}^\star$ & $\bF[\cyc{6}\times\cyc{2},\cyc{6}]$ & $J(E_6){}^\star$    & $\bE[\dih{6}]$ \\
 $J(D_2){}^\star$ & $\bF[\dih{2}\times\cyc{2},\dih{2}]$ & $F$ &  $\bD[\cyc{1}]$  \\
 $J(D_3){}^\star$  & $\bF[\dih{6},\dih{3},\HH]$ & $F_a$   & $\bD[\cyc{2},\R\times\C]$ \\
 $J(D_4){}^\star$  & $\bF[\dih{4}\times\cyc{2},\dih{4}]$ & $F_{ab}$  & $\bD[\cyc{2},\R\times\R]$  \\
 $J(D_6){}^\star$  & $\bF[\dih{6}\times\cyc{2},\dih{6}]$ & $F_{ac}{}^\star$  & $\bD[\cyc{4}]$  \\
 $J(T){}^\star$  & $\bF[\alt{4}\times\cyc{2},\alt{4}]$& $F_{a,b}{}^\star$  & $\bD[\dih{2}]$ \\
 $J(O){}^\star$  & $\bF[\sym{4}\times\cyc{2},\sym{4}]$& $G_{1,3}$  & $\bC[\cyc{1}]$ \\
 $C_{2,1}{}^\star$  & $\bF[\cyc{2},\cyc{1},\M_2(\R)]$  & $N(G_{1,3}){}^\star$  & $\bC[\cyc{2}]$ \\
 $C_{4,1}$  & $\bF[\cyc{4},\cyc{2}]$ & $G_{3,3}{}^\star$  & $\bB[\cyc{1}]$  \\
 $C_{6,1}{}^\star$   & $\bF[\cyc{6},\cyc{3},\M_2(\R)]$ & $N(G_{3,3}){}^\star$ &  $\bB[\cyc{2}]$  \\
 $D_{2,1}{}^\star$  & $\bF[\dih{2},\cyc{2},\M_2(\R)]$  & $\USp(4){}^\star$  & $\bA[\cyc{1}]$\\\hline
\end{tabular}
\vspace{6pt}
\caption{Sato--Tate groups and Galois endomorphism types of abelian surfaces} \label{table:STgroups}
\end{center}
\end{table}

There are $6$ possibilities for the connected component of the identity of the Sato--Tate group of an abelian surface $A$ defined over $k$. As customary, we will denote it by $\ST(A)^0$. The dictionary between Sato--Tate groups and Galois endomorphism types establishes that Sato--Tate groups with identity component isomorphic to the unitary group $\Unitary(1)$ of degree $1$ correspond to \emph{absolute Galois endomorphism type}~$\bF$, that is, to abelian surfaces that are $\Qbar$-isogenous to the square of an elliptic curve with CM by a quadratic imaginary field $M$. A Galois endomorphism type of absolute type~$\bF$ consists of the data
\begin{equation}\label{equation: data GT}
\bF[\Gal(K/k),\Gal(K/kM),\mathbb B],
\end{equation}
where $K/k$ is the minimal extension over which all the endomorphisms of $A$ are defined, and $\mathbb B$ is $\HH$ or $\M_2(\R)$ depending on whether there exists a subextension $K_0/k$ of $K/k$ such that $\End(A_{K_0})\otimes \R$ is isomorphic to $\HH$ or not. We remove $\Gal(K/kM)$ from the notation in case that it coincides with $\Gal(K/k)$, and similarly we avoid writing $\mathbb B$ when the other data is enough to uniquely determine the Galois endomorphism type.

\subsection{A finiteness result}

The first application to the theory of Sato--Tate groups is now an immediate corollary of Theorem~\ref{theorem: absurfQ}. 

\begin{theorem}\label{theorem: finnum} Among the $34$ possibilities for the Sato--Tate group of an abelian surface $A$ defined over $\Q$, the $18$ with identity component isomorphic to $\Unitary(1)$ only occur among the set (of cardinality at most $51$) of $\Qbar$-isogeny classes of abelian surfaces over $\Q$ that are $\Qbar$-isogenous to the square of an elliptic curve defined over $\Qbar$ with CM by $M$ in $\Mm^1 \cup \Mm^2\cup \Mm^{2,2}$. More precisely, for such an abelian surface~$A$, the set of possibilities for $M$ provided that $\ST(A_M)/\ST(A_M)^0\simeq G$ is contained in~$\Mm(G)$.
\end{theorem}
\begin{table}[h]
\begin{center}
\setlength{\extrarowheight}{0.5pt}
\vspace{6pt}
\begin{tabular}{c|c}
\hline
S--T groups & $\Mm(\Gal(K/M))$\\\hline
 $C_{2,1}$ & $\Mm^1$\\
 $J(C_2), D_{2,1}$ & $\Mm^1\cup \Mm^2$ \\
 $C_{6,1},D_{3,2}$ & $\Mm^1$\\
 $J(C_4),D_{4,2}$ & $\{\Q(\sqrt{-1}),\Q(\sqrt{-2})\}\cup \Mm^2$\\
 $J(C_6),D_{6,2}$ & $\{\Q(\sqrt{-3})\}\cup \Mm^2$\\
 $J(D_2),D_{4,1}$ & $\Mm^1\cup \Mm^2\cup \Mm^{2,2}$\\
 $J(D_3),D_{6,1}$ & $\Mm^1\cup \Mm^2$\\
 $J(D_4)$ & $\{\Q(\sqrt{-1}),\Q(\sqrt{-2})\}\cup \Mm^2\cup \Mm^{2,2}$\\
 $J(D_6)$ & $\{\Q(\sqrt{-3})\}\cup \Mm^2\cup \Mm^{2,2}$\\
 $J(T),O_1$ & $\Mm^1\setminus \{\Q(\sqrt{-7})\}$\\
 $J(O)$ & $\{\Q(\sqrt{-1}),\Q(\sqrt{-2})\}\cup \Mm^2\setminus \{\Q(\sqrt{-15}),\Q(\sqrt{-35}),\Q(\sqrt{-51}),\Q(\sqrt{-115})\}$\\\hline
\end{tabular}
\vspace{6pt}
\caption{Possibilities for the field $M$ depending on the Sato--Tate group.}\label{table: MM}
\end{center}
\end{table}

\begin{proof}
We have already mentioned that an abelian surface whose Sato--Tate group has identity component isomorphic to $\Unitary (1)$ has absolute Galois endomorphism type~$\bF$. Observe that
$$
\Gal(K/M)\simeq \ST(A_M)/\ST(A_M)^0\simeq  \ST(A)^{ns}/\ST(A)^{ns,0}\,.
$$ 
The first isomorphism is \cite[Prop. 2.17]{FKRS12}. For the second, as well as for the definition of $\ST(A)^{ns}$, see the paragraph following \cite[Rem. 4.10]{FKRS12}. One concludes by readily checking (for example from \cite[Table 2]{FKRS12}) that the Sato--Tate groups $G$ on the first column and the $i$th row of Table~\ref{table: MM} are precisely those for which $G^{ns}/G^{ns,0}$ is isomorphic to the finite group appearing in the first column and the $i$th row of Table~\ref{table: M}. 
\end{proof}

\subsection{A number field for all genus 2 Sato-Tate groups}\label{section: allST}

The second application and the main result of the present work is the following. 

\begin{theorem}\label{theorem: main}
Set 
$$
k_0:=\Q(\sqrt{-40},\sqrt{-51},\sqrt{-163},\sqrt{-67},\sqrt{19\cdot 43},\sqrt{-57})\,.
$$
Then, $N_{\mathrm{ST},2}(k_0)=52$. That is, there exist $52$ abelian surfaces defined over $k_0$ realizing each of the $52$ possible Sato--Tate groups of abelian surfaces defined over number fields.
\end{theorem}

Hereafter, we will be concerned with the proof of the previous theorem. We will assemble our constructions of abelian surfaces into six main families. Before proceeding to the proof, we make the following remark.

\begin{remark}\label{remark: nofields}
Many of the examples shown in \cite[Table 11]{FKRS12} satisfy that $K$ contains either $\Q(\sqrt{-1})$, $\Q(\sqrt{-2})$, or $\Q(\sqrt{-3})$. In order that these examples also serve as such after base change to~$k_0$, we would rather have~$k_0$ not to contain any of the three latter fields.
As mentioned in the introduction, there are examples of abelian surfaces with Galois endomorphism types $\mathbf{F}[\sym 4\times \cyc 2, \sym 4]$ or $\mathbf{F}[\sym 4, \alt 4]$ (equiv. Sato--Tate groups $J(O)$ or $O_1$) with $M=\Q(\sqrt{-2})$, and we will use them to realize these Galois endomorphism types over $k_0$ via base change. 
 While it is possible to construct an abelian surface over a number field $k$ with Galois endomorphism type $\mathbf{F}[\sym 4]$ and $M=\Q(\sqrt{-1})$ (by applying the below twisting construction to the appropriate representation of $G_8$, for example), this has the undesired consequence that then $k$ is forced to contain $M=\Q(\sqrt{-1})$.
In order to realize $\mathbf{F}[\sym 4]$, we will alternatively use the construction in \S\ref{section: an example} which produces an abelian surface with Galois endomorphism type $\mathbf{F}[\sym 4]$ and $M=\Q(\sqrt{-40})$. This forces~$k_0$ to contain $\Q(\sqrt{-40})$, but as we will see this has rather innocuous effects. 
\end{remark}

\textbf{The restriction of scalars construction.} As mentioned in the above remark, the abelian surface $A$ over $k=\Q(\sqrt{-40})$ defined in \S\ref{section: an example} has Sato--Tate group $O$. Since $K\cap k_0=k$, it follows from \cite[Prop. 2.17]{FKRS12} that the Sato--Tate group of the base change $A_{k_0}$ remains the same.

\textbf{Base change constructions.} We consider the set of curves $\Sigma_\Q$ of \cite[Table 11]{FKRS12} which are defined over $k=\Q$. Note that $\Sigma_\Q$ has cardinality~$34$. We have used the software Sage\footnote{In \url{https://github.com/xguitart/sato-tate} the interested reader can find the Sage script that we used.} to check that for each $C\in \Sigma_\Q$ the number field~$K$ attached to $\Jac(C)$ satisfies that $K\cap k_0=\Q$. It then follows that\footnote{Note that, since $k_0$ contains neither $\Q(i)$, nor $\Q(\sqrt{-3})$, nor $\Q(\sqrt{-2})$, no other curve of \cite[Table 11]{FKRS12} can be base changed to $k_0$ in such a way that the Sato--Tate group is preserved.} $\ST(\Jac(C))\simeq \ST(\Jac(C)_{k_0})$.

\textbf{Twisting constructions.} We will construct examples of abelian surfaces realizing the Sato--Tate groups $C_n$ for $n\in \{1,2,3\}$, $D_n$ for $n\in \{2,3\}$, and $T$ by using the twisting procedure of \cite{MRS07} or \cite[\S2]{Mil72}. For our purposes, a particular case of this construction will suffice: given an elliptic curve $E$ defined over $k_0$ with CM by an order $\mathcal O$ in a quadratic imaginary field $M\subseteq k_0$, a finite Galois extension $L/k_0$, and an Artin representation $\varrho\colon \Gal(L/k_0)\ra \GL_2(\mathcal O)$ with coefficients in $\mathcal O$, there is an abelian surface $A:=E\otimes \varrho$ defined over $k_0$ such that $A_L\sim E_L^2$ and there is an isomorphism of $G_{k_0}$-modules
\begin{equation}\label{equation: twist}
\End(A_\Qbar) \simeq \varrho \otimes \varrho^* \otimes \End(E_\Qbar)\,,
\end{equation}  
where $\varrho^*$ denotes the contragredient representation of $\varrho$ (\cite[Prop. 1.6. i)]{MRS07}). For example, we may take $E$ to be an elliptic curve with CM by $M:=\Q(\sqrt{-163})\subseteq k_0$ and defined over $k_0$. It then follows from \eqref{equation: twist} that the minimal extension $K/k_0$ over which all the endomorphisms of $A$ are defined is the extension cut out by $\varrho \otimes \varrho^*$. For each
$$
H\in \{\cyc{{n}}\text{ for }n\in \{1,2,3\}, \dih{{n}}\text{ for }n\in \{2,3\}\text{, and }\alt 4\}
$$
take a Galois extension $L/k_0$ such that $\Gal(L/k_0)$ be, respectively, isomorphic to
$$
\tilde H\in \{\cyc{{2n}}\text{ for }n\in \{1,2,3\}, \dih{{2n}}\text{ for }n\in \{2,3\}\text{, and }\mathrm{B_T}\}\,.
$$
Take a faithful rational Artin representation $\varrho$ of degree $2$ of $\tilde H$. By inspection of the character table of $H$ one can compute the trace of $\varrho\otimes \varrho^*$ and see that in all the cases the kernel $N$ of $\rho\otimes \varrho^*$ has order $2$, and that the quotient $\tilde H /N$ is isomorphic to $H$;  thus, the field cut by $\rho\otimes \varrho^*$ has Galois group isomorphic to $H$. It follows from the description given in \eqref{equation: data GT} that the Galois endomorphism type of~$A$ is $\mathbf{F}[H]$. According to Table~\ref{table:STgroups}, we have realized the desired Sato--Tate groups.

\begin{remark} One may try to make similar constructions to realize the groups $C_4$, $D_4$, $C_6$, $D_6$, and $O$ over $k_0$. However, in these cases, the corresponding group $\tilde H$ does not possess a faithful rational degree~$2$ representation, but rather one with coefficients in the ring of integers of the fields $\Q(i)$ or $\Q(\sqrt{-2})$. Taking a representation with coefficients in these fields would force $E$ to have CM by them, and thus $k_0$ to contain them. This is a possibility that we have excluded in Remark~\ref{remark: nofields}.
\end{remark}

\textbf{Cardona and Quer's constructions.} We will construct curves whose Jacobians have Sato--Tate groups $C_4$, $D_4$, $C_6$, $D_6$, and $J(C_3)$, respectively. For this we will recall results of Cardona and Quer parametrizing $\Q$-isomorphism classes of genus $2$  curves with prescribed automorphism groups. If $y^2=f(x)$, with $f(x)\in k[x]$, is a defining equation for a genus $2$ curve $C$ defined over $k$, there is an injective group homomorphism
$$
\Aut(C_\Qbar) \hookrightarrow \GL_2(\Qbar)
$$
that sends an automorphism 
$$
(x,y)\mapsto \left(\frac{mx+n}{px+q},\frac{mq-np}{(px+q)^3}y\right),\quad \text{where $m,n,p,q\in \Qbar$,}
$$
to the matrix $\begin{pmatrix}
m & n\\
p & q
\end{pmatrix}$. From now on, we will use matrix notation to write automorphisms of genus $2$ curves.

The next proposition encompasses weakened versions of \cite[Prop. 4.3]{CQ07}, \cite[Prop. 4.9]{CQ07}, and \cite[Thm. 7.4.1]{Car01} which are enough for our purposes.

\begin{proposition}\label{proposition: Cardona}
Let $u,v\in \Q^*$ and $s,z\in \Q$.
\begin{enumerate}[i)]
\item If $1-z^2u=s^2uv$, then the curve given by the affine equation
\begin{equation}\label{equation: D4a}
\begin{array}{lll}
C_{\dih 4}\colon y^2 &= &(1+2uz)x^6-8suvx^5+v(3-10uz)x^4+\\[4pt]
& &+v^2(3-10uz)x^2+8suv^3x+v^3(1+2uz)
\end{array}
\end{equation}
has automorphism group isomorphic to $\dih 4$. It is generated by the matrices
\begin{equation}\label{equation: D4b}
U=\begin{pmatrix}
\alpha & \beta \\
\beta / v & -\alpha
\end{pmatrix} ,\qquad
V=\begin{pmatrix}
0 & -\sqrt{v}\\
1/\sqrt{v} & 0
\end{pmatrix},
\end{equation}
where
\begin{equation}\label{equation: D4c}
\alpha=\sqrt{\frac{1-z\sqrt u}{2}},\qquad \beta=\sqrt{\frac{v(1+z\sqrt u)}{2}}.
\end{equation}
\item If $u^3-z^2=3s^2v$, the curve given by the affine equation
\begin{equation}\label{equation: D6a}
\begin{array}{lll}
C_{\dih 6}\colon y^2&=&27(u+2z)x^6-324svx^5+27v(u-10z)x^4+360sv^2x^3\\[4pt]
&&+9v^2(u+10z)x^2-36sv^3x+v^3(u-2z)
\end{array}
\end{equation}
has automorphism group isomorphic to $\dih 6$. It is generated by the matrices
\begin{equation}\label{equation: D6b}
U=\frac{1}{\sqrt u}\begin{pmatrix}
\alpha & \beta \\
3\beta / v & -\alpha
\end{pmatrix} ,\qquad
V=\frac{1}{2}\begin{pmatrix}
1 & \sqrt{v}\\
-3/\sqrt{v} & 1
\end{pmatrix},
\end{equation}
where $\alpha, \beta \in \Qbar$ are such that
\begin{equation}\label{equation: D6c}
\alpha^3-\frac{3u}{4}\alpha-\frac{z}{4}=0,\qquad \alpha^2+3\frac{\beta^2}{v}= u.
\end{equation}
\item If $u^3-z^2=3s^2v$, the curve given by the affine equation
\begin{equation}\label{equation: 2D12a}
\begin{array}{lll}
C_{2\dih 6}\colon y^2 & = & 27zx^6-162svx^5-135vzx^4+\\[4pt]
&& + 180sv^2x^3+45v^2zx^2-18sv^3x-v^3z
\end{array}
\end{equation}
has automorphism group isomorphic to\footnote{The group $2 \dih 6$ is a certain double cover of $\dih 6$. Its GAP identification number is $\langle 24,8\rangle$.} $2\dih 6$. It is generated by the matrices
\begin{equation}\label{equation: 2D12b}
U=\frac{1}{\sqrt u}\begin{pmatrix}
\alpha & \beta \\
3\beta / v & -\alpha
\end{pmatrix} ,\qquad
V=\frac{\sqrt{-3}}{2}\begin{pmatrix}
1 & -\sqrt{v}/3\\
1/\sqrt{v} & 1
\end{pmatrix},
\end{equation}
where $\alpha, \beta \in \Qbar$ are such that
\begin{equation}\label{equation: 2D12c}
\alpha^3-\frac{3u}{4}\alpha-\frac{z}{4}=0,\qquad \alpha^2+3\frac{\beta^2}{v}= u.
\end{equation}
\end{enumerate}
\end{proposition}

\begin{remark}
We wish to warn the reader of a minor misprint in the work of Cardona and Quer. In \cite[Prop. 3.5]{CQ07} the lower left entry of matrix $U$ is missing a factor of $3$ (compare with \eqref{equation: D6b} above); we have introduced a similar correction to the second equation of both \eqref{equation: D6c} and \eqref{equation: 2D12c}.
\end{remark}

\begin{remark}\label{remark: autos}
It is well known and easy to show that an automorphism group structure of the type $\dih 4$, $\dih 6$ or $2\dih 6$ on a genus $2$ curve induces a decomposition up to $\Qbar$-isogeny of its Jacobian as the square of an elliptic curve $E$. In the $2\dih 6$ case, $E$ has automatically CM by $\Q(\sqrt{-3})$. On \cite[p. 112]{Car01}, one can find the values of the parameter $u\in \Q^*$ for which $E$ has CM in the $\dih 4$ and $\dih 6$ cases. For example, in the former case, for $u=81/320$, $E$ has CM by $\Q(\sqrt{-40})$; in the latter case, for $u=4/17$, $E$ has CM by $\Q(\sqrt{-51})$.
\end{remark}

\begin{lemma}\label{lemma: fieldK}
Let $C$ denote either $C_{\dih 4}$, or $C_{\dih 6}$, or $C_{2\dih 6}$. Suppose that $z\not =0$ and that $\Jac(C)$ is $\Qbar$-isogenous to the square of an elliptic curve with CM by $M$. Then the minimal extension of $\Q$ over which all the endomorphisms of $\Jac(C)$ are defined is $K=M(\alpha,\beta,\sqrt{u},\sqrt{v})$.
\end{lemma}

\begin{proof}
By Remark~\ref{remark: autos}, $K$ is the composition $K'$ of $M$ and the minimal field over which all the automorphisms of $C$ are defined. The case $\dih 4$ is then immediate. For the remaining two cases, we need to check that $K'=M(\alpha/\sqrt u,\beta /\sqrt u, \sqrt v)$ agrees with the expression for $K$ given in the statement. But this follows from the fact that 
$$
\alpha=\left(\alpha^2-\frac{3u}{4} \right)^{-1}\frac{z}{4}\in K',
$$
since $\alpha^2= u(\alpha/ \sqrt u)^2 \in K'$.   
\end{proof}
On Table \ref{table: a few ex}, $C$ denotes one of the curves $C_{\dih 4}$, $C_{\dih 6}$, or $C_{2\dih 6}$ for the choice of parameters $s=1$, $u$ and $z$ as specified on the second and third columns, and $v$ as determined by the constraints of Proposition~\ref{proposition: Cardona}. The fourth and fifth columns are computed using Lemma~\ref{lemma: fieldK}. Together with Remark~\ref{remark: autos}, they imply all but the last row of the last column. 
\begin{table}[h]
\begin{center}
\begin{tabular}{rrrrrr}
$C$ & $u$ & $z$ & $K\cap k_0$ & $\Gal(Kk_0/k_0)$ & $\ST(C_{k_0})$\\\hline
$C_{\dih 4}$ & $81/320$ & $1$ & $\Q(\sqrt{-40})$ & $\dih 4$ & $D_4$\\
$C_{\dih 4}$ & $81/320$ & $-16/9$ & $\Q(\sqrt{-40})$ &  $\cyc 4$ & $C_4$\\
$C_{\dih 6}$ & $4/17$ & $1$ & $\Q(\sqrt{-51})$ &  $\dih 6$ & $D_6$ \\
$C_{\dih 6}$ & $4/17$ & $-5/4$ & $\Q(\sqrt{-51})$& $\cyc 6$ & $C_6$ \\
$C_{2\dih 6}$ & $19$ & $19/2$ & $\Q(\sqrt{-57})$ & $\cyc 6$ & $J(C_3)$\\\hline
\end{tabular}
\vspace{6pt}
\caption{A few examples from Cardona's parametrizations.}\label{table: a few ex}
\end{center}
\end{table}

Suppose that $C$ is as specified in the last row of Table~\ref{table: a few ex}. To verify the bottom right entry, it suffices to show that $\End(\Jac(C_{K_0}))\otimes \R \simeq \HH$, where $K_0=\Q(\sqrt{-57},\alpha)$. This can be easily deduced from the fact that
$$
\Aut(C_{K_0})\simeq \mathrm{Dic}_{12},
$$
where $\mathrm{Dic}_{12}$ is the Dicyclic group of order $12$ (with GAP identification number $\langle 12,1\rangle$). This is a group whose faithful representation of degree 2 has Frobenius--Schur index $-1$ and thus it can not be embedded in $\GL_2(\R)$.

\textbf{Sutherland's additional examples.} Besides the examples showed on \cite[Table 11]{FKRS12}, the wide search performed in \cite{FKRS12}, yielded examples that will be useful to realize $J(C_1)$ and $C_{4,1}$ over $k_0$. We thank Drew Sutherland for pointing out to us the two curves in this paragraph.

Consider the curve
$$
C\colon y^2 = 3x^6 + 16x^5 - 15x^4 - 15x^2 - 16x + 3\,.
$$
We first show that the minimal extension of $\Q$ over which all of the endomorphisms of $\Jac(C)$ are defined is $K=\Q(\sqrt{-10},\sqrt{-2})$.
This follows from the fact that $K$ is minimal with the property
$$
\Aut(C_K)\simeq \GL_2(\Z/3\Z)\,.
$$
Indeed, any genus $2$ curve $C$ such that $\Aut(C_\Qbar)\simeq \GL_2(\Z/3\Z)$ is $\Qbar$-isomorphic to $y^2=x^5-x$.
Since $\Gal(Kk_0/k_0)\simeq \cyc 2$, in order to show that $\ST(\Jac(C_{k_0}))=J(C_1)$, all we need to check is that $\End(\Jac(C_{\Q(\sqrt{-10})})\otimes \R\simeq \HH$. But this follows easily from
$$
\Aut(C_{\Q(\sqrt{-10})})\simeq \SL_2(\Z/3\Z). 
$$
Indeed, $\SL_2(\Z/3\Z)$ contains the group of quaternions $Q$, whose rational faithful degree 2 representation has Frobenius--Schur index $-1$. Thus $\SL_2(\Z/3\Z)$ can not be embedded in $\GL_2(\R)$.

Consider now the curve
$$
C\colon y^2 + (x^3 + x^2 + x + 1)y = -x^5 - 2x^4 - x^3 - 2x^2
$$ 

This is the genus 2 curve \cite[\href{http://www.lmfdb.org/Genus2Curve/Q/40000/e/200000/1}{Genus 2 Curve 40000.e.200000.1}]{LMFDB}. To show that the minimal extension of $\Q$ over which all of the endomorphisms of $\Jac(C)$ are defined is $K=\Q[x]/(x^8-2x^6+4x^4-8x^2+16)$, we again check that $K$ is minimal with the property that
$$
\Aut(C_K)\simeq \GL_2(\Z/3\Z)\,.
$$

Since $\Gal(Kk_0/k_0)\simeq \cyc 4$, in order to show that $\ST(\Jac(C_{k_0}))=C_{4,1}$, all we need to check is that $\End(\Jac(C_{K_0})\otimes \R\simeq \HH$, where $K_0$ is the index $2$ subfield of $K$ given by $\Q[x]/(x^4 + 2x^3 + 4x^2 + 8x + 16)$. But this follows easily from
$$
\Aut(C_{K_0})\simeq \SL_2(\Z/3\Z), 
$$
and the same argument used for the previous curve.

\textbf{Product constructions.} For $D=-19, -43, -67, -163$, let $E_{D}$ be an elliptic curve defined over $k_0$ with CM by the quadratic imaginary field of discriminant $D$. A glimpse at \cite[\S4.4.]{FKRS12} should suffice to convince the reader that
$$
\ST(E_{-67}\times E_{-163})=F,\quad \ST(E_{-19}\times E_{-67})=F_a,\quad \ST(E_{-19}\times E_{-43})=F_{ab}.
$$
Finally, one has that $\ST(E_{-67}\times E)=G_{1,3}$, where $E$ is any elliptic curve defined over $k_0$ without CM.

\end{document}